\documentclass[amsthm]{elsart}
\usepackage{yjsco}
\usepackage{algorithm}
\usepackage{algorithmic}
\usepackage{amsfonts}
\usepackage{amsmath}
\usepackage{epsf,epsfig,floatflt,subfigure}

\newcommand{\Gr}{Gr\"obner }

\def \pol{{\rm poly}}

\def\deg{{\mbox {\rm deg}}}

\def\LM{{\mathrm{LM}}}
\def\LC{{\mathrm{LC}}}
\def\LT{{\mathrm{LT}}}

\def\tail{{\mathrm{tail}}}
\def\poly{{\mathrm{poly}}}
\def \LT{{\rm LT}}
\def \LM{{\rm LM}}
\def \LC{{\rm LC}}

\def \lcm{{\rm lcm}}

\def\ri{\rangle}
\def\li{\langle}
\def\LM{{\mathrm{LM}}}
\def\LC{{\mathrm{LC}}}
\def\LT{{\mathrm{LT}}}
\def\L{{\mathcal{L}}}

\def\I{{\mathcal{I}}}
\def\T{{\mathcal{T}}}
\def\S{{\mathrm{Spoly}}}
\def\nf{{\rm{NF}}}
\def\poly{{\rm{poly}}}
\def\anc{{\rm{anc}}}

\def\nsig{{\rm{sig}}}
\def\sig{{\mathcal {S}}}

\newcommand{\lex}{\mathop{\mathrm{lex}}\nolimits}
\newcommand{\alex}{\mathop{\mathrm{alex}}\nolimits}
\newcommand{\degrevlex}{\mathop{\mathrm{degrevlex}}\nolimits}

\newenvironment{algorithm1}[1]
{
	\begin{center}
		{\bf Algorithm #1} \\
    \begin{tabular}{|p{130mm}|} \hline
}
{
	\\ \hline
	\end{tabular}
	\end{center}
}

\newenvironment{subalgorithm1}[1]
{
	\begin{center}
		{\bf Subalgorithm #1} \\
    \begin{tabular}{|p{130mm}|} \hline
}
{
	\\ \hline
	\end{tabular}
	\end{center}
}

\begin{document}
\begin{frontmatter}

\title{Involutive Bases Algorithm Incorporating F$_5$ Criterion}

\author{Vladimir P. Gerdt}
\address{Laboratory of Information Technologies\\Joint Institute for Nuclear Research, 141980 Dubna, Russia}
\ead{Gerdt@jinr.ru}
\ead[url]{http://compalg.jinr.ru/CAGroup/Gerdt/}

\author{Amir Hashemi}
\address{Department of Mathematical Sciences, \\Isfahan University of Technology, Isfahan, 84156-83111, Iran}
\ead{Amir.Hashemi@cc.iut.ac.ir}
\ead[url]{www.amirhashemi.iut.ac.ir}

\author{Benyamin M.-Alizadeh}
\address{Young Researchers and Elites club, \\Science and Research Branch, Islamic Azad University, Tehran, 461/15655, Iran}
\ead{B.Alizadeh@math.iut.ac.ir}

\begin{abstract}
Faug\`ere's F$_5$ algorithm \cite{F5} is the fastest known algorithm to compute Gr\"obner bases. It has a signature-based and an incremental structure that allow to apply the  F$_5$ criterion for deletion of unnecessary reductions. In this paper, we present an involutive completion algorithm which outputs a minimal involutive basis. Our completion algorithm has a nonincremental structure  and in addition to the involutive form of Buchberger's criteria it applies the F$_5$ criterion whenever this criterion is applicable in the course of completion to involution. In doing so, we use the G$^2$V form of the F$_5$ criterion developed by Gao, Guan and Volny IV \cite{Gao}. To compare the proposed algorithm, via a set of benchmarks, with the Gerdt-Blinkov involutive algorithm \cite{gerdt0} (which does not apply the F$_5$ criterion) we use implementations of both algorithms done on the same platform in {\tt Maple}.
\end{abstract}


\end{frontmatter}
\section{Introduction}
The most universal algorithmic tool in commutative algebra and algebraic geometry is {\em Gr\"obner basis}. The notion of Gr\"obner basis was introduced and an algorithm for its construction was designed in 1965 by Buchberger in his PhD thesis \cite{Bruno1}. Later on \cite{Bruno2}, he discovered two criteria for detecting some unnecessary, and thus useless, reductions that made the Gr\"obner bases method practical for solving a wide class of problems in commutative  algebra, algebraic geometry and in many other research areas of science and engineering~(see, for example, \cite{Bruno3}). However, that  original Buchberger's algorithm turned out to be too time and space consuming for many polynomial systems of scientific and industrial interest. In 1983, Lazard \cite{Daniel83} proposed a new approach based on the linear algebra techniques  to compute Gr\"{o}bner bases. In 1988, Gebauer and M{\"o}ller \cite{gm} reformulated Buchberger's criteria in an efficient way. In 1999, Faug\`{e}re \cite{F4} designed the F$_4$ algorithm to compute Gr\"{o}bner bases. This algorithm stems from Lazard's approach \cite{Daniel83} and exploits fast linear algebra for manipulation with underlying sparse matrices. It has been efficiently implemented in {\tt Maple} and {\tt Magma}. In 2002, Faug\`{e}re designed F$_5$ algorithm, an  incremental algorithm, based on the F$_5$ criterion \cite{F5}. Ars and Hashemi \cite{Extension} proposed a non-incremental version of F$_5$ by defining a new ordering on the signatures that made F$_5$ independent on the order of input polynomials. A correlation between Buchberger's and Faug\`{e}re's  approaches and methods was carefully analysed by Mora \cite{Mora}. Recently Eder and Perry \cite{Eder} simplified some of the steps in F$_5$ by constructing the reduced Gr\"obner basis at each step of the algorithm. Their algorithm called F$_5$C has been implemented in {\tt Singular} and somewhat optimizes F$_5$. Then Gao, Guan and Volny IV in \cite{Gao} presented G$^2$V; a variant of the F$_5$ algorithm whose structure is simpler than that of F$_5$ and F$_5$C and moreover, according to the benchmarking done by the authors of G$^2$V, the last algorithm may be more efficient than the other two signature-based algorithms. We refer to \cite{sun1,sun2,eder2} and references therein for further information on signature-based algorithms.

Another theory which largely parallels the theory of Gr\"obner bases is the theory of {\em involutive bases}. This theory has its origin in the works of the French mathematician Janet. In the $20$s of the last century, he developed \cite{janet} a constructive approach to analysis of linear and certain quasi-linear systems of partial differential equations based on their completion to involution (see the recent book~\cite{seiler} and references therein). The Janet approach has been generalized to arbitrary polynomially non-linear differential systems by the American mathematician Thomas \cite{thomas}. Based on the involutive methods as they have been presented in Pommaret's book~\cite{pommaret}, Zharkov and Blinkov introduced in \cite{zharkov} the notion of {\em involutive polynomial bases}. The particular form of an involutive basis that they studied is nowadays called {\em Pommaret basis}~\cite{seiler}.

Gerdt and Blinkov \cite{gerdt0} proposed a more general concept of involutive bases for polynomial ideals and designed algorithmic methods to construct these bases. The basic underlying idea of the involutive approach to commutative algebra is to translate the methods originating from Janet's approach into polynomial ideal theory in order to provide an algorithmic method for construction of polynomial involutive bases by combining algorithmic ideas in the theory of Gr\"obner bases and constructive ideas in the theory of involutive differential systems. In doing so, Gerdt and Blinkov \cite{gerdt0} introduced the concept of {\em involutive monomial division}\footnote{Inspired by these results, Apel~\cite{Apel98} put forward a somewhat different concept of involutive division.} and established two criteria to avoid some useless involutive reductions. This led to a computational tool (see the web pages {\tt http://invo.jinr.ru} and {\tt http://wwwb.math.rwth-aachen.de/Janet/} ) which can be considered as an efficient alternative ({\tt http://cag.jinr.ru/wiki/Benchmarking\_for\_polynomial\_ideals}) to the conventional Buchberger's algorithm (note that any involutive basis is also a Gr\"obner basis). Apel and Hemmecke \cite{detecting} discovered that there are two more criteria for detecting unnecessary involutive reductions (see also \cite{gerdt2}) which, in the aggregate with the criteria by Gerdt and Blinkov \cite{gerdt0}, are equivalent to Buchberger's criteria. The first author decribed in \cite{gerdtnew} a computationally efficient algorithm for involutive and Gr\"obner bases using all these criteria. Different versions of involutive algorithms \cite{gerdtnew,gerdt0,zharkov} based on the concept of involutive division by Gerdt and Blinkov have been implemented in {\tt Reduce, Singular, Macaulay2, Maple} and very recently in {\tt CoCoA}. The fastest implementation is the one done in {\tt GINV} \cite{ginv}. For application of involutive bases to commutative algebra and to algebraic-geometric theory of partial differential equations we refer to Seiler's book \cite{seiler}.

The conventional and involutive full normal forms modulo an involutive basis are equal \cite{gerdt0}. Thereby, a natural question that arises is: {\em How to incorporate the F$_5$ criterion into an involutive algorithm?} In the given paper, we answer this question by proposing a new structure for the algorithm described in \cite{gerdtnew}. We shall refer to the last as the Gerdt-Blinkov involutive (GBI) algorithm. Our structure allows to exploit the F$_5$ criterion. For the sake of simplicity,  we use here the G$^2$V version of the criterion. Our new algorithm is not incremental since at each step we must consider the set of multiplicative and nonmultiplicative variables for the whole set of polynomials in the intermediate basis including the input ones (see Section \ref{Inc}). However, by using the signature characterization inherent in any version of the F$_5$ algorithm, we provide an F$_5$-consistent {\em involutive completion} of the input polynomial set. Such a completion applies the F$_5$ criterion as much as possible and ends up with an involutive basis. Then a minimal involutive basis is constructed from the last basis. We have implemented the new algorithm in {\tt Maple} for the Janet division \cite{gerdt0} and for the $\succ_{\alex}$-division \cite{alex}. In order to analyse the new algorithm experimentally, via some benchmarks, and to make its comparison with the GBI algorithm at the same implementation platform, we implemented the last algorithm in {\tt Maple} for both involutive divisions as well.

The paper is organized as follows. Section \ref{sec:1} contains some basic definitions and notations related to the theory of involutive bases, and the algorithm for construction of minimal such bases in its simplest form \cite{gerdtnew}. In Section \ref{F5algo} we present briefly the F$_5$ criterion and its G$^2$V modification. Section \ref{Inc} is devoted to the description of our new involutive algorithm for computing minimal involutive bases. At the end of this section, we give an illustrating example for the proposed algorithm. In Section \ref{Exp} we present our experimental comparison of the new algorithm with the GBI algorithm.

\section{Involutive bases}
\label{sec:1}

In this section, we recall some basics from the theory of involutive bases and briefly describe the algorithm for their construction in its simplest form \cite{gerdtnew}.

Let $K$ be a field and $R:=K[x_1,\ldots,x_n]$ be the polynomial ring in the variables $x_1,\ldots,x_n$ over $K$. Below, we denote a {\em monomial} $x_1^{\alpha_1}\cdots x_n^{\alpha_n}\in R$ by ${\bf x}^\alpha$ where $\alpha=(\alpha_1,\ldots,\alpha_n) \in \mathbb{N}^{n}$ is a sequence of non-negative integers. We shall use the notations $\deg_i({\bf x}^\alpha):=\alpha_i$ and  $\deg({\bf x}^\alpha):=\sum_i \alpha_i$. An {\em admissible} monomial ordering on $R$ is a total order $\prec$ on the set of all monomials such that\\[0.1cm]
  \hspace*{0.5cm} (i) $ 1\prec \mathbf{x}^\alpha\ \text{for all}\ \mathbf{x}^\alpha\neq 1$\,,\ \ (ii) $ \mathbf{x}^\alpha \prec \mathbf{x}^\beta\ \text{implies}\ \mathbf{x}^{\alpha+\gamma} \prec \mathbf{x}^{\beta+\gamma}\  \text{for all}\ \alpha,\beta,\gamma \in \mathbb{N}^{n}$\,.

\noindent
The typical examples of such monomial orderings denoted respectively by $\prec_{\lex}$  and $\prec_{\degrevlex}$ are {\em lexicographical} and {\em degree-reverse-lexicographical}. Given monomials ${\bf x}^\alpha$ and ${\bf x}^\beta$, ${\bf x}^\alpha \prec_{\lex} {\bf x}^\beta$ if the left-most non-zero entry of $\beta-\alpha$ is positive; ${\bf x}^\alpha \prec_{\degrevlex} {\bf x}^\beta$ if $\deg({\bf x^\alpha})>\deg({\bf x^\beta})$ or $\deg({\bf x^\alpha})=\deg({\bf x^\beta})$ and the right most non-zero entry of $\beta-\alpha$ is negative.

Let  ${\I}=\langle f_1,\ldots ,f_k\rangle$ be the {\em ideal} in $R$ generated by the polynomials $f_1,\ldots ,f_k\in R$. Furthermore, let $f\in R$ and $\prec$ be a monomial ordering on $R$. The {\em leading monomial} of $f$ is the greatest monomial (with respect to $\prec$) occurring in $f$, and we denote it by $\LM(f)$. Respectively, the {\em leading term} of $f$ is denoted by $\LT(f)$ and the {\em leading coefficient} by $\LC(f)$. If $F\subset R$ is a set of polynomials, we denote by $\LM(F)$ the set $\{\LM(f) \ \mid \ f\in F\}$. The {\em leading monomial ideal} of ${\I}$ is defined to be $\LM({\I})=\langle \LM(f)\ \mid \ f \in {\I}\rangle.$
A finite set $G\subset {\I}$ is called a {\em Gr\"obner basis} of ${\I}$ if $\LM({\I})=\langle \LM(G) \rangle$. 
We refer to the book by Becker and Weispfenning~\cite{Becker} for more details on Gr\"obner bases.

We recall below the definition of involutive bases. For this purpose, we describe first the notion of an {\em involutive division} \cite{gerdt0} which is a restricted monomial division \cite{gerdtnew} that, together with a monomial ordering, determines properties of an involutive basis. This makes the main difference between involutive bases and Gr\"obner bases. The idea is to partition the variables into two subsets of {\em multiplicative} and {\em nonmultiplicative} variables, and only the multiplicative variables can be used in the divisibility relation.

\begin{defn}(\cite{gerdt0})
An {\em involutive division} ${\L}$ on the set of monomials of $R$ is given, if  for any finite set $U$ of monomials and any $u \in U$, the set of variables is partitioned into the subsets of {\em multiplicative} $M_{\L}(u,U)$ and {\em nonmultiplicative} $NM_{\L}(u,U)$ variables such that the following three conditions hold:
\begin{itemize}
\item $u,v\in U,\ u\cdot {{\L}(u,U)} \cap v\cdot{{\L}(v,U)}\neq \emptyset$ implies $u\in v\cdot {{\L}(v,U)}$ or $v \in u\cdot {{\L}(u,U)}$,
\item $u,v\in U,\ v \in u\cdot{{\L}(u,U)}$ implies ${{\L}(v,U)} \subset {{\L}(u,U)}$,
\item $u \in V$ and $V \subset U$ implies ${{\L}(u,U)} \subset {{\L}(u,V)}$,
\end{itemize}
where ${\L}(u,U)$ denotes  the monoid generated by the variables in $M_{\L}(u,U)$. If $v \in u\cdot{\L}(u,U)$ then  $u$ is called ${\L}$-{\em (involutive) divisor} of $v$ and the involutive divisibility relation is denoted by $u \mid_{\L} v$.
If $v$ has no involutive divisors in a set $U$, then it is called {\em ${\L}$-irreducible} modulo $U$.
\label{invdiv}
\end{defn}

There are involutive divisions based on the classical partitions of variables suggested by Janet \cite{janet} and Thomas \cite{thomas}. In this paper, we are also concerned with the wide class \cite{alex} of involutive divisions determined by a total monomial ordering $\sqsupset$ such that it is either admissible or the inverse of an admissible ordering, and by a permutation $\sigma$ on the indices of variables:
\begin{equation}
 (\ \forall u\in U\ ) \ \ [\ NM_{\sqsupset}(u,U)=\bigcup_{v\in U\setminus \{u\}}NM_{\sqsupset}(u,\{u,v\})\ ]
 \label{pair}
\end{equation}
\begin{equation}
NM_{\sqsupset}(u,\{u,v\}):=\left\lbrace
\begin{array}{l}
\text{ if } u\sqsupset v \text{ or } (u\sqsubset v \wedge v\mid u) \text{ then } \emptyset \\[0.1cm]
\text{ else }\{x_{\sigma(i)}\},\ i=\min\{j\mid \deg_{\sigma(j)}(u)< \deg_{\sigma(j)}(v)\}\,.\\
\end{array}
\right. \label{inv_div}
\end{equation}

{\em Throughout the given paper, ${\L}$ is assumed to be either the division in this class, denoted by $\sqsupset$-division, or the Thomas division}. If we consider the monomial ordering $\sqsupset$ to be the lexicographical ordering, then (\ref{pair})-(\ref{inv_div}) generates the Janet division. As to the Thomas division (denoted by $\T$), it satisfies (\ref{pair}), but unlike (\ref{inv_div}) this division is not linked to a total monomial ordering. Instead,
\begin{equation}
NM_{\T}(u,\{u,v\}):=\{\,x \mid \deg_x(u)<\deg_x(v)\,\}\,.
\label{Thomas}
\end{equation}

\noindent
Now, we define an involutive basis.

\begin{defn}(\cite{gerdtnew})
Let ${\I}\subset R$ be an ideal, $\prec$ a monomial ordering on $R$ and ${\L}$ an involutive division. A finite set $G\subset {\I}$ is an {\em involutive basis} (or {\em ${\L}$-basis}) of ${\I}$ if for any $f\in {\I}$ there is $g\in G$ such that $\LM(g) \mid_{\L} \LM(f)$. If $U\subset R$ is a finite monomial set, then a monomial set $\bar{U}$ is called {\em ${\L}$-completion} of $U\subseteq \bar{U}$ if $\bar{U}$ is an ${\L}$-basis of $\langle U\rangle$.
\label{def_invbase}
\end{defn}

From this definition and from that of a Gr\"obner basis~\cite{Becker,Bruno1} it follows that an involutive basis is a Gr\"obner basis of the ideal that it generates, but the converse is not always true. {\em Noetherianity} of a division ${\L}$~\cite{gerdt0,alex} guarantees the existence of an ${\L}$-basis.

\begin{prop}{\em ({\cite{gerdt0}})}
\label{jan}
 Any ideal has a finite ${\L}$-basis.
\end{prop}

\begin{rem}
By using an involutive division in polynomial rings, we obtain an {\em involutive division algorithm}~\cite{gerdt0}. Given a finite polynomial set $F$ and a monomial ordering, we denote by $\nf_{\L}(f,F)$ the remainder in ${\L}$-division of $f$ by $F$.
\end{rem}

For an involutive division ${\L}$, the following theorem provides an algorithmic characterization of involutive basis for a given ideal that is an involutive analogue of the Buchberger characterization of a Gr\"obner basis.

\begin{thm} {\em(\cite{gerdt0})}
\label{blin}
Let ${\I}\subset R$ be an ideal, $\prec$ a monomial ordering on $R$ and ${\L}$ an involutive division. Then a finite generating set $G \subset {\I}$ is an ${\L}$-basis of ${\I}$ if for each $f\in G$ and each $x\in NM_{\L}(\LM(f),\LM(G))$, we have $\nf_{\L}(x\cdot f,G)=0$.
\end{thm}

\begin{defn}(\cite{gerdt0})
An involutive basis $G$ is called {\em minimal} if for any other involutive basis $\tilde{G}$ of $\li G\ri $ the inclusion $\LM(G)\subseteq \LM(\tilde{G})$ holds. Similarly, a {\em minimal involutive completion} of a monomial set is a subset of any other involutive completion of this set.
\label{def_min}
\end{defn}
\noindent
A minimal involutive basis being monic and {\em ${\L}$-autoreduced}, i.e. satisfying
$$(\forall g\in G)\ [\,g=\nf_{\L}(g,G\setminus\{g\})\,]$$
is {\em unique} for a given ideal and a monomial ordering. Similarly, the minimal monomial involutive completion is uniquely defined.

\begin{prop}
For any $\sqsupset$-division and any finite set $U$ of monomials the following inclusion holds
\begin{equation}
(\,\forall u\in U\,)\ [\,NM_{\sqsupset}(u,U)\subseteq NM_{\T}(u,U)\,]\,,\quad \bar{U}_\sqsupset \subseteq \bar{U}_{\T}
\label{incl}
\end{equation}
where $\bar{U}_{\sqsupset}$ and $\bar{U}_{\T}$ denotes the minimal completion of $U$ for the $\sqsupset$-division and for the Thomas division, respectively.
\label{comparison}
\end{prop}

\begin{pf}
The first inclusion (for nonmultiplicative variables) is an obvious consequence of (\ref{inv_div}) and (\ref{Thomas}). The second inclusion is an immediate consequence of the first one.
\end{pf}

\begin{rem}
The minimal Thomas completion $\bar{U}_{\T}$ of a monomial set $U$ is given~\cite{gerdt0} by
\begin{equation}
\bar{U}_{\T}=\{\,u\in \langle U\rangle \mid deg_i(u)\leq \max\{\,\deg_i(v)\mid v\in U\,\}\,,\  i=1,\ldots,n\,\}\,.
\label{T-completion}
\end{equation}

\end{rem}

\noindent
Based on Theorem \ref{blin}, one can design an algorithm to compute involutive bases. We recall here the {\sc InvBas} ({\sc Inv}olutive{\sc Bas}is) algorithm  from \cite{gerdtnew}. The algorithm outputs a {\em minimal} involutive basis.

As it is emphasized in \cite{gerdtnew}, in comparison to the algorithm of the second paper in \cite{gerdt0}, another selection strategy is used in {\sc InvBas} that optimizes the displacement done in the {\bf for}-loop (lines 8-11). By this strategy, a polynomial is chosen in line 5 whose leading monomial has no proper divisors among the leading monomials in $Q$. However, the below  version of the GBI algorithm is still not efficient in practice, since it repeatedly processes the same prolongations and does not apply any criterion to avoid superfluous reductions.

\vskip 0.2cm
\begin{algorithm1}{{\sc InvBas}
\label{invbas}}
\begin{algorithmic}[1]
\INPUT $F$, a set of polynomials; ${\L}$, an involutive division; $\prec$, a monomial ordering
\OUTPUT a minimal involutive basis of $\li F \ri$
\STATE Select $f \in F$ with no proper divisor of $\LM(f)$ in $\LM(F)$;
\STATE $G:=\{f\}$;
\STATE $Q:=F \setminus G$;
\WHILE {$Q \ne \emptyset$}
     \STATE Select and remove $p \in Q$ with no proper divisor of $\LM(p)$ in $\LM(Q)$;
     \STATE $h:=\nf_{\L}(p,G)$;
     \IF {$h\ne 0$}
           \FOR {$g\in G$ such that $\LM(g)$ is properly divisible by $\LM(h)$}
                  \STATE {$Q:=Q\cup \{g\}$;}
                  \STATE $G:=G \setminus \{g\}$;
           \ENDFOR
           \STATE {$G:=G\cup \{h\}$;}
           \STATE {$Q:=Q\cup \{x\cdot g \ |\ g\in G, x\in NM_{\L}(\LM(g),\LM(G))\}$;}
     \ENDIF
\ENDWHILE
\RETURN ($G$)
\end{algorithmic}
\end{algorithm1}
\vskip 0.2cm

The improved version of the algorithm which is also presented in \cite{gerdtnew} clears away the repeated processing of prolongations and applies the involutive form of Buchberger's criteria. Below, we suggest one more algorithmic improvement which, in addition to the indicated ones, admits application of the F$_5$ criterion for deletion of useless reductions.

\section{F$_5$ criterion and G$^2$V algorithm}
\label{F5algo}

This section presents the theory behind the F$_5$ algorithm. After recalling some notations and definitions, we state the main theorem of \cite{F5} which is the cornerstone of the F$_5$ algorithm.  To this end, we consider the recent paper \cite{eder2} due to Eder and Perry. Finally, we present briefly the G$^2$V algorithm \cite{Gao} further developed in \cite{Gao2}.

Let ${\I}=\li f_1,\ldots ,f_k \ri\subset R=K[x_1,\ldots,x_n]$ be the ideal in $R$ generated by the polynomials $f_1,\ldots ,f_k$,
$R^k$ be a free $R$-module of rank $k$ and ${\mathbf{e}_1},\ldots ,{\mathbf{e}_k}$ be its canonical basis. For the sake of simplicity assume that each $f_i$ is monic, i.e. $\LC(f_i)=1$.
A {\em module monomial} is an element in $R^k$ of the form $m{\mathbf{e}_i}$ where $m\in R$ is a monomial.
Let us denote the set of all module monomials by $\mathbf{M}$.
A monomial ordering $\prec$ on $R$ can be extended to a {\em module  monomial ordering} on $\mathbf{M}$, denoted by $<$, as follows
$$\sum^{j}_{i=1}{g_i\cdot {\mathbf{e}_i}}< \sum^{\ell}_{i=1}{h_i\cdot {\mathbf{e}_i}} \ \ {\rm if} \  \  \left\{\begin{array}{lcl}
j>\ell &{\rm and}&h_\ell \ne 0 \ {\rm or} \\[-0.2cm]
j=\ell & {\rm and} & \LM(g_j)\prec \LM(h_j),\ g_j\cdot h_j\neq 0.
\end{array}\right. $$


\begin{defn}
Let $f\in \I$ and $1\leq j\leq k$ be the smallest integer for which $h_j\ne 0$ in $f=\sum_{i=1}^kh_if_i$. Then, the module monomial $\LM(h_j)\,{\mathbf{e}_j}$ is called a {\em natural signature} of $f$. Under the assumption, $\LM(h_j){\mathbf{e}_j}$ is the greatest module monomial w.r.t. $<$ occurring  in $\sum_{i=1}^kh_i{\mathbf{e}_i} \in R^k$. Denote by $\nsig(f)$ the set of all natural signatures of $f$. As one can easily see, $<$ is a well-ordering on $\mathbf{M}$ and thus $\nsig(f)$ has a unique  minimal element, which is called the {\em minimal natural signature} of $f$.
\end{defn}

Let $f\in R$ and  $m{\mathbf{e}_i} = \nsig(f)$. Then the pair $(m{\mathbf{e}_i}, f)\in R^k\times R$ is called a {\em labelled polynomial} associated to $f$. We shall denote the set of all labelled polynomials by ${\bf L}$, and from this point on we shall write $\sig(f)$ for the minimal signature associated to $f$ by the signature based algorithm under consideration.

\begin{defn}
Let $r = (m{\mathbf{e}_i}, f)\in {\bf L}$. Then $f$, $m{\mathbf{e}_i}$ and  $i$ are, respectively, called  the {\em polynomial part, signature} and {\em index} of $r$ and we denote them by $\poly(r)$, $\sig(r)$ and  index$(r)$. We define also $\LM(r), \LC(r)$ and $\LT(r)$ as $\LM(f), \LC(f)$ and $ \LT(f)$, respectively. Furthermore, the labelled polynomial $r$ can be multiplied by a monomial $u$ and by an element $c\in K$ of the ground field in accordance to the rules $ur = (u m {\mathbf{e}_i}, uf)$ and $cr = (m {\mathbf{e}_i}, cf)$.
\label{multrules}
\end{defn}
Denote by $\psi$ the following map:
$$
\begin{array}{cccl}
  \psi:& R^k &\rightarrow &R\\[-0.2cm]
&(g_1,\ldots, g_k)& \mapsto& g_1f_1 + \cdots + g_kf_k.
\end{array}
$$
A labelled polynomial $r = (\sig(r), \pol(r))$ is called {\em admissible} if there exists ${\bf g} \in R^k$ such that  $\psi({\bf g}) = \pol(r)$ and  the greatest module monomial w.r.t. $<$, occurring  in ${\bf g}$ is $\sig(r)$. It is easy to check that with the above notations, $um{\mathbf{e}_i} \in \nsig(uf)$, and hence the multiplication rules in Definition~\ref{multrules} preserve the admissibility.
Let us explain now the reduction algorithm used in ${\bf L}$. In this case, the reduction               is more restrictive than the usual polynomial reduction. If $(m{\mathbf{e}_i},f)$ is reducible by  $(m'{\mathbf{e}_j},g)$; i.e. $t\LM(g)=\LM(f)$ for a monomial $t\in R$, then one of the following cases holds.
\begin{itemize}
\item (safe reduction):  If $tm'{\mathbf{e}_j} < m{\mathbf{e}_i}$, then the reduction is performed as $(m{\mathbf{e}_i},f-tg)$.
\item (unsafe reduction): If $tm'{\mathbf{e}_j} \geq m{\mathbf{e}_i}$, then the signature is changed at the reduction step, and such a reduction is not performed.
\end{itemize}
This reduction algorithm provides an alternative definition (cf. ~\cite{Becker}, Definition 5.59) of standard representation for labelled polynomials.
\begin{defn}
Let $P\subset {\bf L}$ be a finite set, and $r,t\in {\bf L}$ with $f:=\poly(r)\ne 0$. We say that $r$ has a {\em $t$-representation} w.r.t. $P$ if $f=\sum_{p_i\in P}{h_i\poly(p_i)}$ $(h_i\in {cR})$ and for all $p_i\in P$ with $\poly(p_i)\ne 0$ the following relations hold:
$$ \LM(h_i) \LM(p_i)\preceq \LM(t) \ \ {\rm and}  \ \  \LM(h_i) \sig(p_i)\le \sig(r).$$
This property is written as $r=\mathcal{O}_P(t)$. We shall also write $r=o_P(t)$ if there exists a labelled polynomial $t'\in A$ satisfying $ \sig(t')\le \sig(t)$ and $\LM(t')\prec \LM(t)$ such that  $r=\mathcal{O}_P(t')$. A $t$-representation of $r$ is called {\em standard}  if $\LM(t)=\LM(r)$.
\end{defn}
\noindent

To state a Buchberger-like criterion  for labelled polynomials (Theorem \ref{thm: sr for L}),  we need to define the $S$-polynomial of two labelled polynomials.  Suppose that $f,g\in R$ are two polynomials. The conventional {\em $S$-polynomial} of $f$ and $g$ is defined to be
$$\S(f,g)=\frac{\lcm(\LM(f),\LM(g))}{\LT(f)}f-\frac{\lcm(\LM(f),\LM(g))}{\LT(g)}g\,.$$
Let now $r=(\sig(r),f)$ and $s=(\sig(s),g)$ be two admissible labelled polynomials with $u=\lcm(\LM(f),\LM(g))/\LM(f),v=\lcm(\LM(f),\LM(g))/\LM(g)$. If $v\sig(s)< u\sig(r)$, we define the {\em labelled $S$-polynomial} of $r$ and $s$ to be $\S(r,s)=(u\sig(r),\S(f,g))$. Otherwise, we do not consider such an $S$-polynomial, see \cite{eder2} for more details.

\begin{thm}{\em (\cite{eder2})}
\label{thm: sr for L}
Let  ${\I}=\li f_1,\ldots ,f_k \ri\subset R$ and let $G\subset {\bf L}$ be a finite set of admissible labelled polynomials such that
\begin{itemize}
\item for all $1\leq i\leq k$ there exists $r_i\in G$ such that $f_i=\poly(r_i)$,
\item for each pair $(r_i,r_j)\in G\times G$ either $u_i\sig(r_i)=u_j\sig(r_j)$, $\S(r_i,r_j)=0$ or $\S(r_i,r_j)=o_G(u_s\cdot r_s)$ where $u_s={\rm lcm}(\LM(r_i),\LM(r_j))/\LM(r_s)$ for $s\in \{i,j\}$.
\end{itemize}
Then the set $\{\poly(r)\ | \ r \in G\}$ is a \Gr basis of ${\I}$.
\end{thm}

All algorithms to compute a Gr\"obner basis which take ``signatures'' into account  are called {\em signature-based} algorithms, see \cite{eder2}. Most of these algorithms, like Faug\`ere's F$_5$ algorithm \cite{F5} are described {\em incrementally} to apply the F$_5$ criterion. The F$_5$ algorithm computes sequentially the Gr\"obner bases of the ideals
$$\li f_k\ri,\li f_{k-1},f_{k}\ri,\ldots ,\li f_1,\ldots ,f_k\ri.$$
\begin{defn}
Let  ${\I}=\li f_1,\ldots ,f_k \ri\subset R$. An admissible labelled polynomial $(m{\mathbf{e}_i},f)$ is called {\em normalized}, if $m\notin \LM(\li f_{i+1},\ldots ,f_{k}\ri)$. A pair $(r_1,r_2)$ of admissible labelled polynomials is {\em normalized} if $u_ir_i$ and $u_jr_j$ are normalized where $u_s={\rm lcm}(\LM(r_i),\LM(r_j))/\LM(r_s)$ for $s\in \{i,j\}$.
\end{defn}
\begin{thm}{\em (F$_5$ criterion)}
By the assumptions of Theorem \ref{thm: sr for L}, the $S$-polynomial of a non-normalized pair $(r_i,r_j)\in G\times G$ has a standard representation w.r.t. $G$, and therefore, it is superfluous.
\end{thm}
\begin{pf}
See \cite{F5}, Theorem 1.
\end{pf}
The F$_5$ algorithm, in its original form \cite{F5}, as the first signature-based algorithm is rather difficult to understand and to implement. Gao, Guan and Volny IV \cite{Gao} (see also \cite{Gao2}) presented an algorithm called G$^2$V which may be considered as a version of the F$_5$ algorithm. This version seems to be simpler and  more efficient than the original F$_5$ (cf. the benchmarking in \cite{Gao}).  By this reason, we use G$^2$V to apply the F$_5$ criterion in construction of involutive bases.

To explain the structure of G$^2$V, assume that $G=\{g_1,\ldots,g_s\}$ is a Gr\"obner basis of $\li f_{i+1},\dots,f_k\ri$ where $1\leq i\leq k-1$. Our goal is to compute a Gr\"obner basis of $\li f_{i},\dots,f_k\ri$. Much like F$_5$C \cite{Eder}, G$^2$V uses the reduced Gr\"obner basis obtained at the preceding step of the algorithm. The structure of polynomials in G$^2$V is slightly different from that mentioned before. As a labelled polynomial, G$^2$V considers a pair $(m,f)\in R^2$ where $m$ is a monomial and $f\in R$ is a polynomial. The monomial $m$ is called the {\it signature} of the pair. A pair $(m,f)$ is {\em admissible}, if there exists $p \in R$ with $pf_i\equiv f$ {\rm mod} $G$ such that $\LM(p)=m$. This is consistent with  the representation of this labelled polynomial $(m{\mathbf{e}_i},f)$ in F$_5$.

Hereafter, we shall consider only admissible labelled polynomials and omit the term `admissible'.
Initially, G$^2$V considers the labelled polynomials $(0,g_1),\ldots,(0,g_s)$ and $(1,g)$ where  $g$ is the normal form of $f_i$ modulo $G$. Then, it creates the {\em J(oint)-pairs} of $(1,g)$ and of the other labelled polynomials. By definition, the J-pair of two labelled polynomials $(m,f)$ and $(m',f')$ is the labelled polynomial of the form $(t m,t f)$ with $t=\lcm(\LM(f),\LM(f'))/\LM(f)$ and  $t'=\lcm(\LM(f),\LM(f'))/\LM(f')$ satisfying $\LM(t' f')\preceq \LM(t f)$.

G$^2$V takes the J-pair with the smallest signature and repeatedly performs only {\it regular top-reductions} of this pair as long as such regular top-reduction is possible. A labelled polynomial $(m,f)$ is {\em top-reducible} by another labelled polynomial $(m',f')$ if there is a monomial $t \in R$ such that $\LM(f)=t \LM(f')$ and $t m' \prec m$. The corresponding top-reduction is defined as
$(m,f)-t(m',f')=(m,f/\LC(f)-t f'/\LC(f'))$.
If $t m'=m$, then the reduction is called {\it super}, otherwise it is called {\it regular}.

Let $(m,f)$ be the result of reduction of a labelled polynomial. If $f\ne 0$, then $(m,f)$ is added to the current \Gr basis, and the new J-pairs are formed. If $f=0$, then G$^2$V uses $m$ to delete useless J-pairs. Namely, a labelled polynomial $(m',f')$ can be discarded \cite{Gao} if $m'\ne 0$ and $m \mid m'$. In doing so, this kind of reduction is considered in \cite{Gao} as a super top-reduction too (see also \cite{Gao2}).  We shall also say that $(m',f')$ is super top-reducible by $(m,0)$ when $m \mid m'$.

Now, we state and prove the theorem which stems from the results of paper \cite{Gao} and provides the correctness of the G$^2$V algorithm (cf. \cite{Gao2}, Theorem 2.3) and thereby the correctness of applying the F$_5$ criterion in G$^2$V.

\begin{thm} [G$^2$V form of F$_5$ criterion]
\label{g2vthm}
Let $G$ be a Gr\"obner basis of $\li f_{i+1},\dots,f_k\ri$ and $G'$ the output of G\,$^2$V. Then the set $T=\{f \ |\  (m,f) \in G'\}$ is a Gr\"obner basis of $\li f_{i},\dots,f_k\ri$, if for each J-pair $(t m,t f)$ of the elements in $G'$ one of the following conditions holds:
\begin{enumerate}
\item  $(t m,t f)$ reduces to zero on the regular top-division by $G'$.
\item  $t m$ is multiple of an element in $\LM(G)$.
\item $(t m,t f)$ reduces to $(t m,g)$ on the regular top-division by $G'$ so that $(t m,g)$ is no longer regular top-reducible by $G'$, and $(t m,g)$ is super top-reducible by  $G'$.
\item $m'\mid t m$ where  $m'$ is the signature of a labelled polynomial $(m',0)$ obtained during the computation of $G'$.
\end{enumerate}
\end{thm}
\begin{pf}
We must prove that for each J-pair $(t m,t f)$ of the elements in $G'$, $t f$ has a {\em standard representation} w.r.t. $T$ (see \cite{Becker}, Theorem 5.64). The reducibility of the pair to zero yields immediately the standard representation for $(t m,t f)$. If the second condition holds, the pair is not normalized and we refer to \cite{F5}, Theorem 1 (see also \cite{Mora}). To prove the third item, suppose that $(t m,t f)$ reduces to $(t m,g)$ on the regular top-division by $G'$, and $(t m,g)$ is super top-reducible by $(m',f')\in G'$, i.e. $t m=s m'$ and $t \LM(g)=s \LM(f')$ for some monomial $s\in R$. Note that  $(t m,g)$ may be super top-reducible by an element $(m',0)$. In this case, which corresponds to the forth condition, we consider $f'=0$. From admissibility of $(t m,g)$ and $(m',f')$ we can write $g=p_i f_i+\sum_{j=i+1}^k p_j f_j$ and $f'=p'_i f_i+\sum_{j=i+1}^k p'_j f_j$ where $p_j,p'_j\in R$, $p_i,p'_i$ are monic, $\LM(p_i)=t m$ and $\LM(p'_i)=m'$. In that follows we denote $p-\LT(p)$ by $\tail(p)$ for a polynomial $p$. Thus,
\begin{eqnarray*}
g&=&t m f_i+\tail(p_i)f_i+\sum_{j=i+1}^k p_j f_j\\
&=& s m' f_i+ \tail(p_i) f_i+\sum_{j=i+1}^k  p_j f_j\\
&=&s f'-s\cdot \tail(p'_i) f_i-\sum_{j=i+1}^k s p'_j f_j+\tail(p) f_i+\sum_{j=i+1}^k  p_j f_j\\
&=&s f'+( \tail(p_i)-s\cdot \tail(p'_i)) f_i+\sum_{j=i+1}^k (p_j-s p'_j) f_j\,.
\end{eqnarray*}
This implies that polynomial $g-s f'$ can be written as
$$( \tail(p_i)-s\cdot \tail(p'_i)) f_i+\sum_{j=i+1}^k (p_j-s p'_j) f_j\,.$$
In accordance to the choice of monomial ordering made in Section \ref{sec:1}, the signature of labelled polynomial with this polynomial part is strictly less than $t m$. Therefore, $t f$ has a standard representation w.r.t. $T$.
\end{pf}

\begin{rem}
In the G$^2$V algorithm, there are usually many J-pairs with the same signature. In this case Gao, Volny IV and Wang \cite{Gao2} claim  that one can just store one J-pair $(m,f)$ whose polynomial part $f$ has the minimal leading monomial and discard all other J-pairs with signature $m$ (see \cite{Gao2}, Theorem 2.3 for more details).
\end{rem}

\begin{rem}
Until recent papers by Huang \cite{huang}, by Pan, Hu and Wang \cite{pan} and by Galkin\cite{galkin} there has not been strong evidence for termination of the F$_5$ (and respectively G$^2$V) algorithm. Based on the results of~\cite{huang}, in their preprint \cite{Gao2} Gao, Volny IV and Wang formulated the termination condition as compatibility of the monomial ordering $\succ$ with the module monomial ordering $>$:
$$ \LM(f) \mathbf{e}_i > \LM(g) \mathbf{e}_i\ \ \text{if and only if}\ \ \LM(f)\succ \LM(g)\,.$$
Note that the orderings we use satisfy this compatibility condition.
\label{termination}
\end{rem}

\section{Involutive completion algorithm}
\label{Inc}

In this section, we present an algorithm which applies the F$_5$ criterion in computing involutive bases.  Let ${\I}=\li f_1,\ldots,f_k\ri \subset R=K[x_1,\ldots,x_n]$ be an ideal, $\prec$ be a monomial ordering on $R$, and ${\L}$ be an involutive division. By an {\em incremental algorithm} for construction of an involutive basis for ${\I}$, we mean that based on the sequential construction of involutive bases for the ideals $\li f_k\ri, \li f_{k-1},f_k\ri,\ldots,\li f_1,\ldots,f_k\ri$  and on the use of an involutive basis of $\li f_{i+1},\ldots,f_k\ri$ for construction of a basis for $\li f_i,\ldots,f_k\ri$. The main obstacle to such incremental construction is that at each step we must manipulate with the multiplicative and nonmultiplicative variables for the leading monomials in the whole set of intermediate polynomials. To get over this obstacle we design an algorithm not in the incremental style, i.e. we do not compute {\em completely} the involutive basis for the corresponding intermediate ideals $\li f_i,\ldots,f_k\ri$ $(1<i<k)$. Instead, having added a new polynomial to the intermediate polynomial set, we use it to update all the previous steps by taking into account the leading monomial of the new polynomial.  This provides, after termination of the algorithm, that the obtained basis is an involutive one of the input ideal.  We call this process {\em completion} of the input polynomial set {\em to involution} in accordance to the conventional terminology used in the theory of involution~\cite{seiler}. For this purpose, we use a signature based selection strategy w.r.t. the module monomial ordering defined in Section \ref{F5algo}. More precisely, we select a polynomial whose signature is minimal w.r.t. $<$, and therefore a chosen  polynomial (from a set of polynomials to process) has the maximal index. We shall call this strategy the {\em G\,$^2$V selection strategy}.

We describe now the structure of polynomials that is used in our new algorithm. To apply the F$_5$ criterion, we must rely on the structure of labelled polynomials defined in Section \ref{F5algo}. In doing so, we present a labelled polynomial $r$ in the form $r=(m\cdot \mathbf{e}_i,f)$ where $\sig(r)=m\cdot \mathbf{e}_i$ and $\pol(r)=f$, whereas in the algorithm implementation (see Section \ref{Exp}) the G\,$^2$V form of labelled polynomials is used.

In \cite{gerdt0}, the involutive form of Buchberger's criteria were presented to avoid a part of unnecessary reductions (see also \cite{detecting,gerdtnew}). In order to use these criteria in our algorithm and to avoid the repeated prolongations, we add extra information to the labelled structure of polynomials. This extra information is similar to that used in~\cite{gerdtnew}. So, we recall the following definition.

\begin{defn}
Let $F \subset R \setminus \{0\}$ be a finite set of polynomials. An {\em ancestor} of a polynomial $f \in F$, denoted by $\anc(f)$, is a polynomial $g \in  F$ of the smallest $\deg(\LM(g))$ among those satisfying $\LM(f) = u\cdot \LM(g)$ where $u$ is either the unit monomial or a power product of nonmultiplicative variables for $\LM(g)$ such that $\nf_{\L}(f-u\cdot g,F\setminus\{f\})=0$ if $f\neq u\cdot g$.
\label{ancestor}
\end{defn}

This additional information on the history of prolongations allows to avoid some unnecessary reductions by applying the adapted Buchberger's criteria (below we discuss these criteria after presenting the main algorithm). Now, to each polynomial $f$, we associate a quadruple $p=(m\cdot \mathbf{e}_i,f,g,V)$ where $\poly(p)=f$ is the polynomial part of $p$, $\sig(p)=m\cdot \mathbf{e}_i$ is its signature part, $\anc(p)=g$ is the ancestor of $f$ (or of $p$) and $NM_{\L}(p)=V$ is the set of nonmultiplicative variables of $f$ such that the corresponding prolongations of the polynomial have been already constructed. By keeping this set, one can avoid the repeated treatment of nonmultiplicative prolongations. If $P$ is a set of quadruples, we denote by $\poly(P)$ the polynomial set $\{\,\poly(p) \ \mid \ p\in P\,\}$. Where no confusion can arise, we may refer to a quadruple $p$ instead of $\poly(p)$, and vice versa.

Our main algorithm {\sc InvComp}, given a finite set $F$ of polynomials, an admissible monomial ordering and an involutive division ${\L}$,  computes an ${\L}$-basis of $\li F\ri$ by completion of the input set with non-zero polynomials that are computed in the course of the algorithm.

\begin{algorithm1}{\sc InvComp \label{invcomp}}
\begin{algorithmic}[1]
\INPUT $F=\{f_1,\ldots,f_k\}$, a finite set of nonzero polynomials; ${\L}$, an involutive division; $\prec$, a monomial ordering such that $\LM(f_1)\succeq \LM(f_2)\succeq \cdots \succeq \LM(f_k)$
\OUTPUT a minimal ${\L}$-basis of $\li F \ri$
\STATE $ArxivLM:=[[\LM(f_1)],\ldots,[\LM(f_k)]]$;
\STATE $T:=\{({\bf e}_k,f_k,f_k,\emptyset)\}$;\quad $Q:=\emptyset$;
\IF{$k>1$}
     \STATE $Q:=\{(\mathbf{e}_i,f_i,f_i,\emptyset)  \ |\ i=1,\ldots,k-1\}$;
\ENDIF
\WHILE {$Q \ne \emptyset$}
     \STATE Select and remove $p=(m\cdot \mathbf{e}_i,f,g,V) \in Q$ with minimal signature w.r.t. $<$
     \STATE $h:=${\sc RegNormalForm}$(p,T,{\L},\prec)$;
     \IF {$h = 0$}
          \IF{$\LM(f)=\LM(g)$}
            \STATE $T:=T\setminus \{\,t\in T \mid anc(t)=g\,\}$;
          \ENDIF
     \ELSE
       \STATE $ArxivLM[i]:=ArxivLM[i] \cup \{\LM(h)\}$;
       \IF {$\LM(f)\ne \LM(h)$}
               \STATE $T:=T \cup \{(m\cdot \mathbf{e}_i,h,h,\emptyset)\};$
       \ELSE
                \STATE $T:=T \cup \{(m\cdot \mathbf{e}_i,h,g,V)\};$
       \ENDIF
       \FOR {$q\in T$ and $x\in NM_\L(\LM(\poly(q)),\LM(\poly(T))) \setminus NM_{\L}(q)$}
           \STATE $Q:=Q \cup \{(x\cdot \sig(q),x\cdot \poly(q),\anc(q),\emptyset)\}$;
           \STATE $NM_{\L}(q):=NM_{\L}(q)\cup \{x\}$;
           \IF{$\LM(h)\mid_{\L}\LM(\poly(q))$}
                 \STATE $u:=\frac{\LM(\poly(q))}{\LM(h)}$;
                 \IF{$q - u\cdot h\neq 0$}
                    \STATE $Q:=Q \cup \{(u\cdot m\cdot \mathbf{e}_i,q-u\cdot h,q-u\cdot h,\emptyset)\}$;
                 \ENDIF
          \ENDIF
       \ENDFOR
   \ENDIF
\ENDWHILE
\STATE $G:=${\sc MinBas}$(\poly(T),{\L},\prec)$;
\RETURN ($G$) 
\end{algorithmic}
\end{algorithm1}
\vskip 0.2cm

\noindent
The involutive completion is performed in the {\bf while}-loop (lines 6-31). Noetherianity of the input involutive division ${\L}$ and its constructivity~\cite{gerdt0,alex} provide the existence of an ${\L}$-basis by processing only the nonmultiplicative prolongations~\cite{gerdtnew,gerdt0}. In its line 7, the algorithm {\sc InvComp} uses the G\,$^2$V selection strategy for an element in $Q$ to be processed in the {\bf while}-loop. The ${\L}$-reductions of the chosen polynomial are performed by the algorithm {\sc RegNormalForm} invoked in line 8.

For a constructive involutive division a minimal \Gr basis is a well-defined subset of any involutive basis \cite{gerdt0}.  The {\bf while}-loop outputs an involutive basis $\poly(T)$, and its minimal involutive subset is computed in line 32 by the subalgorithm {\sc MinBas}.

$ArxivLM$ is a global variable. At the initialization step of {\sc InvComp}) to the $i$-th element of $ArxivLM$ the leading monomial of the input polynomial $f_i$ is assigned. Then this element of $ArxivLM$ collects (line 14 of {\sc InvComp}) the leading monomials of those computed polynomials which belong to the ideal $\li f_i,\ldots,f_k\ri$ and whose signature basis vector is $\mathbf{e}_i$. Furthermore, the set $Q$ of the polynomials to process, is another global variable and we update it in {\sc RegNormalForm} invoked in line 8. This algorithm returns, by performing regular ${\L}$-reductions, an ${\L}$-{\em regular normal form} of the polynomial under processing. Indeed, a labelled polynomial $p_1=(m_1\cdot {\bf e}_{i_1},f_1,g_1,V_1)$  is  ${\L}$-regular top-reducible by $p_2=(m_2\cdot {\bf e}_{i_2},f_2,g_2,V_2)$ if $\LM(f_1)$ is ${\L}$-divisible by $\LM(f_2)$, and  for the $t \in R$ such that $\LM(f_1)=t\cdot \LM(f_2)$ the relation $t\cdot m_2\cdot {\bf e}_{i_2} < m_1\cdot {\bf e}_{i_1}$ holds. Remark that all polynomials occurring in the computation are assumed to be monic. Thus, the corresponding top-reduction of the leading terms is given by $p_1-t\cdot p_2$. Recall that if $t\cdot m_2\cdot {\bf e}_{i_2} = m_1\cdot {\bf e}_{i_1}$ the ${\L}$-reduction is {\em super}, otherwise it is {\em regular}.

\vskip 0.2cm
\begin{subalgorithm1}{\sc RegNormalForm}
\label{RegNormalForm}
\begin{algorithmic}[1]
\INPUT $p$, a quadruple;  $T$ a finite set of quadruples; ${\L}$, an involutive division; $\prec$, a monomial ordering
\OUTPUT ${\L}$-regular normal form of $\poly(p)$ modulo $T$
\STATE $h:=\poly(p)$;
\STATE $r:=0$;
\WHILE {$h\ne 0$}
   \IF {$(\,\exists q\in T\,)\ [\,\LM(\poly(q))\mid_{\L}\LM(h)\,] $}
   \STATE Select such $q$;
   \STATE $u:=\frac{\LM(h)}{\LM(\poly(q))}$;
     \IF { $u\sig(q)\le \sig(p)$}
         \IF {$\LT(h)=\LT(\poly(p))$ and {\sc Criteria}($p,q$)}
              \RETURN ($0$)
       \ELSE
              \STATE $h:=h-\poly(q)\frac{\LT(h)}{\LT(\poly(q))}$;
       \ENDIF
     \ELSE
        \STATE $Q:=Q \cup \{(u\cdot \sig(q),h-u\ \poly(q),h-u\ \poly(q),\emptyset)\}$;
     \ENDIF
   \ELSE
     \STATE $r:=r+\LT(h)$;
     \STATE $h:=h-\LT(h)$;
   \ENDIF
\ENDWHILE
\RETURN ($r$)
\end{algorithmic}
\end{subalgorithm1}
\vskip 0.2cm
\noindent
The subalgorithm {\sc RegNormalForm} performs the ${\L}$-regular top-reductions and also some involutive tail reductions. Moreover, this  subalgorithm detects some unnecessary reductions by applying the F$_5$ criterion and the involutive form of Buchberger's criteria.

\noindent
In \cite{gerdt0}, the involutive consequences $C_1$ and $C_2$ of Buchberger's first and second criteria, respectively,  were presented to avoid some unnecessary reductions (see Lemma \ref{buch}). Then, Apel and Hemmecke in \cite{detecting} discovered two more criteria (see also \cite{gerdtnew}) that in the aggregate with $C_2$ are equivalent to Buchberger's chain criterion. The computer experimentation done by the first author and Yanovich \cite{yanovich} revealed that these two criteria, being applied when the criteria  $C_1$ and $C_2$ are not applicable, often (for not very large examples) slowdown computation of involutive bases. That is why, in the given paper we use only the criteria $C_1$  and $C_2$.

In the subalgorithm {\sc Criteria} with the input polynomials $p$ and $q$, the Boolean expression Buch($p,q$) is true if at least one of the criteria $C_1$ or $C_2$ is applicable, and false otherwise.
The correctness of applying $C_1$ or $C_2$ in our algorithm under the G\,$^2$V selection strategy is provided by the following lemma (cf. \cite{gerdt0,detecting}).

\begin{lem}
\label{buch}
  Let ${\I}\subset R$ be an ideal, $\prec$ be a monomial ordering on $R$ and ${\L}$ be an involutive division. Let $P:=\poly(T)\subset {\I}$ be the current polynomial set computed in the course of {\sc InvComp}, and $f=\poly(p)\in {\I}$ be the polynomial selected in line 7 of the last algorithm. Then $\nf_{\L}(f,P)=0$ if there exists $q \in P$ with $\LM(q) \mid_{\L} \LM(f)$  satisfying one of the following conditions:
\begin{enumerate}
\item[$\hspace*{0.5cm}(C_1)$] $\LM(\anc(f))\cdot \LM(\anc(q))=\LM(f)$,
\item[$(C_2)$] $\lcm(\LM(\anc(f)),\LM(\anc(q)))$ is a proper divisor of $\LM(f)$.
\end{enumerate}
\end{lem}

\begin{proof}
  Suppose that $f$ and $q$ satisfy $C_1$. Then, the following two cases are possible.
  \begin{enumerate}
  \item[$(i)$] $\lcm(\LM(\anc(f)),\LM(\anc(q)))$  is a proper divisor of $\LM(f)$, i.e there is a monomial $s\ne 1$ such that $\LM(f)=s\cdot \lcm(\LM(\anc(f)),\LM(\anc(q)))$.
  \item[$(ii)$] $\LM(f)=\lcm(\LM(\anc(f)),\LM(\anc(q)))$.
  \end{enumerate}
  In the case $(i)$ without loss of generality and in accordance to Definition \ref{ancestor} we may let $f=u\cdot \anc(f)$, $q=v\cdot \anc(q)$ and $\LT(f)=t\cdot \LT(q)$ for some monomials $u$ and $v$ and term $t$. Thus,
  $$f-t\cdot q=u\cdot \anc(f)-t\cdot v\cdot \anc(q)=s\cdot (u'\cdot \anc(f)-v'\cdot \anc(q))$$
  where $u=s\cdot u'$ and $v\cdot t=s\cdot v'$ for some monomials $u'$ and $v'$. Furthermore, since $\lcm(\LM(\anc(f)),\LM(\anc(q)))$  is a proper divisor of $\LM(f)$, $\sig(\anc(f))\le \sig(f)$ and $\sig(\anc(q))\le \sig(q)$, the polynomial $u'\cdot \anc(f)-v'\cdot \anc(q)$ has a signature strictly less than $\sig(f)$. Hence, $u'\cdot \anc(f)-v'\cdot \anc(q)$ has been processed before $f$ (by the G\,$^2$V selection strategy). Therefore, $u'\cdot \anc(f)-v'\cdot \anc(q)$ has a standard representation w.r.t. $P$. This implies that $f-t\cdot q$ has also a standard representation w.r.t. $P$ and thus, the \Gr normal form of $f$ modulo $P$ is zero. Furthermore, the G\,$^2$V selection strategy and the extension of $Q$ done in lines 26 of {\sc InvCom} and 14 of {\sc RegNormalForm}  guarantee that the \Gr normal form of $p$ modulo $P$ coincides with the ${\L}$-normal form (by the {\em partial involutivity} of $P$ up to the monomial $\LM(f)=\LM(t)\cdot \LM(q)$ \cite{gerdt0}).  Therefore, $\nf_{\L}(f,P)=0$.

  In the case $(ii)$, by Buchberger's first criterion, the equality $$\lcm(\LM(\anc(f)),\LM(\anc(q)))=\LM(\anc(f))\cdot \LM(\anc(q))$$ implies that $u'\cdot \anc(f)-v'\cdot \anc(q)$ has a standard representation w.r.t. $P$, and the equality $\nf_{\L}(f,P)=0$ is proved by exactly the same reasoning as used above for the case $(i)$.

If $f$ and $q$ satisfy $C_2$, then Buchberger's chain criterion is applicable ~\cite{gerdt0} to $f-t\cdot q$, and the proof is similar to that done for $C_1$.
\end{proof}

\vskip 0.5cm
\begin{subalgorithm1}{\sc  Criteria}
\label{Criteria}
\begin{algorithmic}[1]
\INPUT $p=(m\cdot \mathbf{e}_i,f,g,V)$ and $q$, quadruples
\OUTPUT true if one of the criteria listed in Theorem \ref{g2vthm} or  Lemma \ref{buch} holds, and false otherwise\\[0.1cm]
       \IF {$\frac{\LM(\poly(p))}{\LM(\poly(q))}\sig(q)= \sig(p)$ or Buch($p,q$)}
              \RETURN (true)
       \ENDIF
       \FOR {$j$ from $i+1$ to $k$}
               \IF {$t$ divides $m$ for some $t\in ArxivLM[j]$}
                    \RETURN (true)
               \ENDIF
       \ENDFOR
\RETURN (false)
\end{algorithmic}
\end{subalgorithm1}
\vskip 0.5cm

\begin{rem}
  Let $p=(m\cdot \mathbf{e}_i,f,g,V)$ be a quadruple that is selected in line 7 of the main algorithm for its processing in the {\bf while}-loop. If $m$ is divisible by a monomial in $ArxivLM[j]$ for $j>i$, then we can eliminate $p$ by the F$_5$ criterion in accordance to the structure of the algorithm and Theorem \ref{g2vthm}.
\end{rem}

\noindent
The particular type of reduction that we use in the {\sc RegNormalForm} subalgorithm makes inapplicable the displacement done by the {\bf for}-loop (lines 8-11) in the algorithm {\sc InvBas} (see also \cite{gerdtnew}) to construct a minimal involutive basis. Instead, we use the algorithm {\sc MinBas} based on the following trivial lemma to extract a minimal involutive basis (see Definition~\ref{def_min}) from a given involutive basis (cf. \cite{gerdt2}).

\begin{lem}
\label{lem1}
Let $G\subset R$ be an involutive basis, $\prec$ a monomial ordering on $R$ and ${\L}$ an involutive division. Then, $G$ is a minimal involutive basis  if and only if $\LM(G)$ is a minimal monomial involutive basis.
\end{lem}

\begin{algorithm1}{\sc MinBas \label{minimal }}
\begin{algorithmic}[1]
\INPUT $H$, an ${\L}$-basis of $\li H \ri$; ${\L}$, an involutive division; $\prec$, a monomial ordering
\OUTPUT $G$, a minimal ${\L}$-basis of $\li H \ri$
       \STATE Select and remove a polynomial $p \in H$ with no proper divisor of $\LM(p)$ in $\LM(H)$
       \STATE $G:=\{p\}$;
       \WHILE { $H\neq \emptyset$ }
          \STATE Select a polynomial $h \in H$ without proper divisors of $\LM(h)$ in $\LM(H)$
          \STATE $H:=H \setminus \{h\}$;
          \IF { $\not \exists\, g\in G$ s.t. $\LM(h)\in {\L}(\LM(g),\LM(G))$}
            \STATE $G:=G\cup \{h\}$
          \ENDIF
       \ENDWHILE
\RETURN ($G$)
\end{algorithmic}
\end{algorithm1}
\vskip 0.2cm
\noindent
To prove the correctness of the suggested algorithm we discuss first the correctness of the involutive form of Theorem \ref{g2vthm}. For this purpose, we need the following proposition which is an obvious consequence of Definition~\ref{def_invbase} and the fact that the involutive divisibility implies the conventional divisibility (cf.~\cite{gerdt0}).

\begin{prop}
\label{thmgerdt}
  Let $F\subset R$ be a finite set, $\prec$ a monomial ordering on $R$ and ${\L}$ an involutive division. If $F$ is an involutive basis, then the equality of the conventional and ${\L}$-normal forms modulo $F$ and $\prec$ holds for any  normal form algorithm.
\end{prop}

\noindent
This proposition immediately implies that the conditions in Theorem \ref{g2vthm} can be rewritten in terms of involutive reductions.

\begin{cor}
\label{corgerdt}
Let $G$ be an involutive basis for $\li f_{i+1},\dots,f_k\ri$ where $1\leq i\leq k-1$. Let  $G'$ be the set of all labelled polynomials computed by any signature-based algorithm (like {\sc InvComp} algorithm) for computing an involutive basis for $\li f_{i},\dots,f_k\ri$. Then $G'$ is an involutive basis for $\li f_{i},\dots,f_k\ri$, if for each $(m\cdot \mathbf{e}_i,f,g,V)\in G'$ and each $x\in NM_{\L}(\LM(f),\LM(G'))$ one of the following conditions holds:
\begin{enumerate}
\item  $(x\cdot m\cdot \mathbf{e}_i,x\cdot f,g,V)$ reduces to zero on ${\L}$-regular top-division by $G'$.
\item  $x\cdot m$ has an ${\L}$-divisor in $\LM(G)$.
\item $(t\cdot m,t\cdot f)$ reduces to $(t\cdot m,g)$ on the ${\L}$-regular top-division by $G'$ so that $(t\cdot m,g)$ is no longer ${\L}$-regular top-reducible, and $(t\cdot m,g)$ is ${\L}$-super top-reducible by $G'$.
\item $m'\mid_{\L} t\cdot m$ where  $m'$ is the signature of a labelled polynomial $(m',0)$ obtained in the computation of $G'$.
\end{enumerate}
\end{cor}

\begin{thm}
\label{main}
  The {\sc InvComp} algorithm outputs a minimal involutive basis of the polynomial ideal generated by the input polynomial set.
\end{thm}

\begin{pf}
{\em Correctness}. Lemma \ref{buch} and Corollary \ref{corgerdt} guarantee the correctness of subalgorithm {\sc Criteria} invoked in line 6 of the subalgorithm {\sc RegNormalForm} when the condition of the {\bf if}-statement (line 4) is true. If this condition is false on account of the signature relation, then all possible intermediate results of the ${\L}$-reduction chain are inserted into $Q$ (line 14). Thereby, the full involutive normal form of the input polynomial $h$ (line 1) modulo $T$ is to be eventually computed and inserted into $T$ whenever this normal form is non-zero.  Apparently, the ideal generated by polynomials in $T\cup Q$ is the loop invariant
\begin{equation}
{\I}:=\langle F \rangle = \langle \,\poly(T\cup Q)\,\rangle\,.
\label{ideal}
\end{equation}

If the algorithm {\sc InvComp} terminates, then $Q=\emptyset$ and all the ${\L}$-nonmultiplicative prolongations of polynomials in $T$ constructed in line 21 have already been processed. Thus, $U:= \LM(\poly(T))$ is ${\L}$-involutive. Moreover, because of the ordering of the input polynomials and the selection strategy for an element in $Q$ to be processed (line 7 of {\sc InvComp}), the elements in $U$ are distinct monomials. Indeed, this selection and the enlargement of $Q$ done in line 26 of {\sc InvComp} and in line 14 of {\sc RegNormalForm} imply
\begin{equation}
  (\,\forall q\in Q\,)\ (\,\forall t\in T\,)\ [\,\sig(q)\geq \sig(t)\,]\,. \label{sig_ineq}
\end{equation}

Therefore, if $\LM(\poly(p))=\LM(\poly(t))$, where $p$ is the quadruple selected in line 7 of {\sc InvComp} and $t\in T$, the ${\L}$-head reduction of $\poly(p)$ by $\poly(t)$ is allowed in {\sc RegNormalForm} since  the {\bf if}-condition of line 4 in the last subalgorithm is true.

\noindent
Now we show that $\langle U\rangle$ generates the leading monomial ideal of $\I$, i.e.
\begin{equation}
\langle U\rangle = \LM(\I).
\label{week_inv}
\end{equation}

Let $P:=\poly(T)$ be the intermediate polynomial set contained in $T$ directly before a run of the {\bf while}-loop and let $\tilde{P}$ denotes the polynomial set obtained by the ${\L}$-head autoreduction of $P$. We claim that

\begin{equation}
\LM(\tilde{P})\subseteq \LM(P)\,.
\label{inclusion}
\end{equation}

To prove it, we note first that, in accordance to the initiation step 2, the inclusion (\ref{inclusion}) holds trivially before the very first run of the loop. Then, every enlargement of $H:=\poly(T)$ with $h$ done in line 16 or in line 18 of {\sc InvComp} is attended with insertion of every possible non-zero polynomial obtained by the elementary ${\L}$-head reduction modulo $h$ of a polynomial in $T$ into the polynomial part of $Q$. This insertion is done in line 26 of the {\bf for}-loop (lines 20-29). If such a new element added to $Q$ will again ${\L}$-reduce a polynomial in $T$ at the stage of its selection in line 8, then this will again lead to an extension of ${\poly(Q)}$ with the result of the corresponding (non-zero) elementary reduction, and so on. Finally, after completion of the {\bf while}-loop, for every polynomial $h$ in $H$ its ${\L}$-head normal form will be an element in $H$. This proves the claim.

Now, by the third condition in Definition \ref{inv_div}, a polynomial $f\in\tilde{P}$ cannot give rise  to new ${\L}$-nonmultiplicative variables as a result of ${\L}$-head autoreduction of $P$. Therefore, all nonmultiplicative prolongations of $f$ are ${\L}$-reduced to zero modulo $\tilde{P}$. It follows that $\LM(\tilde{P})$ is an ${\L}$-autoreduced monomial set and $\tilde{P}$ is an involutive basis of  (\ref{ideal}) (cf.~\cite{gerdt0}). This implies the equality (\ref{week_inv}) and shows that $P$ is also an involutive basis of $\I$.

\noindent
Finally, by Lemma \ref{lem1}, the subalgorithm {\sc MinBas} invoked in line 32 of {\sc InvCom} returns a minimal involutive basis as a subset of its input involutive basis.

\noindent
{\em Termination}. First, we note that termination of ${\L}$-reduction and termination of the subalgorithm {\sc Criteria} provide termination of the subalgorithm {\sc RegNormalForm}. Second, in the course of algorithm {\sc InvComp} the intermediate set $T$ can only be enlarged by the insertion of new elements in line 16 or 18. In doing so, the cardinality of the set $Q$ is obviously bounded at every intermediate step of the algorithm. The repeated processing of nonmultiplicative prolongations is excluded by means of the set $NM_{\L}(q)$ associated to every polynomial $q\in \poly(T)$ and used in the {\bf for}-statement of line 20. Recall that $NM_{\L}(q)$ contains all those variables $x\in NM_{\L}(q,\poly(T))$ for which $x\cdot q$ has been already processed. Thus, to prove the termination of the algorithm, it suffices to show that the cardinality of $T$ is bounded, that is, the cardinality of the leading monomial set $U:=\LM(P)$ where $P:=\poly(T)$ is bounded.

There are three alternative variants for the completion of $U$ with $u:=\LM(h)$ where $h:=${\sc RegNormalForm}$(p,T,{\L},\succ)$ and $p\in Q$ is the quadruple selected in line 7 of {\sc InvComp}:
\begin{enumerate}
\item Either $u$ has no ${\L}$-divisors in $U$ or $u$ is ${\L}$-reducible modulo $U$ but the reduction is not allowed by the signature condition (line 4 in {\sc RegNormalForm}).
\item $u$ is ${\L}$-reducible modulo $U$ and $h$ is obtained from $\poly(p)$ by its {\em partial} ${\L}$-head reduction modulo $P$ such that at least one head reduction has been performed. Thus, $\LM(h)\prec \LM(\poly(p)$ and there is $q\in P$ such that $\LM(q)\mid_{\L} u$ but the ${\L}$-head reduction of $h$ by $q$ is not allowed in {\sc RegNormalForm} by the signature condition.
\item $h$ is the {\em full} ${\L}$-head normal form of $\poly(p)$ modulo $P$ and $h\neq 0$.
\end{enumerate}

\noindent
There are finitely many cases to complete $U$ (into a set, say $\bar{U}$) by either  the monomials obtained by processing of the input polynomials or by the monomials which are not in $\langle U\rangle$. The number of last monomials is finite by virtue of Dickson's lemma~\cite{Becker}, and they can only occur in the 3rd of the above variants.

Therefore, it remains to show that there cannot be infinitely many completions of $U$ preserving $\langle U\rangle$ in the case when elements in $Q$ are either nonmultiplicative prolongations of the polynomials in $P$ or
${\L}$-head reductions of such prolongations or (if ${\L}$ is a $\sqsupset$-division with non-admissible $\sqsupset$, e.g. antigraded~\cite{CLO:2005:UAG}) ${\L}$-head autoreductions of the polynomials in $P$. For the Thomas division, the maximal possible number of completions of $U$ is obviously bounded by the cardinality of $\bar{U}_{\T}$, the minimal Thomas completion of $U$ given by (\ref{T-completion}). If ${\L}$ is a $\sqsupset$-division, then, by Proposition \ref{comparison}, the total number of completions also cannot exceed the cardinality of $\bar{U}_{\T}$ in (\ref{T-completion}).
\end{pf}

\begin{cor}
If the input involutive division in algorithm {\sc InvComp} is either Thomas division or $\sqsupset$-division with admissible $\sqsupset$, then the lines 23-28 in the algorithm can be omitted.
\label{cor_mult}
\end{cor}

\begin{pf}
Let $U:=\LM(\poly(T))$ where $T$ is the intermediate set of quadruples in algorithm {\sc InvComp} and let $u,v\in U$ be two monomials such that $u$ is a proper divisor of $v$. From (\ref{pair}) and (\ref{Thomas}) it follows immediately that $u$ cannot $\T$-divide $v$ since there is a variable $x\mid v$ such that $\deg_x(u)<\deg_v(v)$ and, hence $x$ is $\T$-nonmultiplicative for $u$. Consider now $\sqsupset$-division with admissible $\sqsupset$. In this case, $u\sqsubset v$ and the variable $x_{\sigma(i)}$ specified in (\ref{inv_div}) is nonmultiplicative for $u$. Therefore, for both $\T$-and $\sqsupset$-divisions $u$ cannot divide $v$ involutively.
\end{pf}

\begin{cor}
If the input involutive division ${\L}$ in algorithm {\sc InvComp} is $\sqsupset$-division generated by  the total monomial ordering $\sqsupset$ which is {\em antigraded} \cite{CLO:2005:UAG}, then to obtain a minimal ${\L}$-basis from the output $T$ of the {\bf while}-loop in the main algorithm {\sc InvComp} one can perform ${\L}$-head autoreduction of $\poly(T)$.
\end{cor}

\begin{pf}
See the proof in \cite{alex} of Theorem 2 and Corollary 2.
\end{pf}

We give now a simple example illustrating the behavior of algorithm {\sc InvComp} for the Janet division.
\begin{exmp}
  Let $f_1=x^2-3/2 y^2,f_2=2 x y+3 y^2$, $F=\{f_1,f_2\}\subset K[x,y]$ and $y\prec_{\lex}x$. Let $p_1=({\bf e}_1,x^2-3/2 y^2,f_1,\{~\})$ and $p_2=({\bf e}_2,2 x y+3 y^2,f_2,\{~\})$. Then, {\it ArxivLM}$=[[x^2],[xy]]$, $T=\{p_2\}$ and $Q=\{p_1\}$. We select and remove $p_1$ from $Q$. \\
$\Rightarrow ${\sc RegNormalForm}$(p_1,T,{\L},\prec)=x^2-3/2y^2$\\
$\Rightarrow T=\{p_2,p_1\}$\\
$\Rightarrow p_3=(x\cdot {\bf e}_2,x (2 x y+3y^2),f_2,\{~\})$\\
$\Rightarrow Q=\{p_3\}$\\
$\Rightarrow $ we select and remove $p_3$ from $Q$, and its normal form modulo $T$ is $2x^2y-9/2y^3$\\
$\Rightarrow T=\{p_2,p_1,p_3\}$ where $p_3=(x\cdot {\bf e}_2,2 x^2 y-9/2 y^3,f_2,\{~\})$\\
$\Rightarrow${\it ArxivLM}$=[[x^2],[x y,x^2 y]]$\\
$\Rightarrow p_4=(y\cdot {\bf e}_1,y (x^2-3/2 y^2),f_1,\{~\})$\\
$\Rightarrow Q=\{p_4\}$\\
$\Rightarrow $ we select and remove $p_4$ from $Q$, and its normal form modulo $T$ is $3/4y^3$\\
$\Rightarrow T=\{p_2,p_1,p_3,p_4\}$ where $p_4=(y\cdot {\bf e}_1,3/4 y^3,f_1,\{~\})$ \\
$\Rightarrow${\it ArxivLM}$=[[x^2,y^3],[x y,x^2 y]]$\\
$\Rightarrow p_5=(x\cdot y\cdot {\bf e}_1,x (3/4 y^3),f_4,\{~\})$\\
$\Rightarrow Q=\{p_5\}$\\
$\Rightarrow $ we select and remove $p_5$ from $Q$\\
$\Rightarrow $ {\sc Criteria}$(p_5,p_2)=$true, because $xy\in \li${\it ArxixLM}$[2]\ri$, and we remove $p_5$ by F$_5$ criterion\\
$\Rightarrow Q=\{~\}$\\
$\Rightarrow G=\{p_2,p_1\}$ is a minimal \Gr basis for $\li F\ri$\\
$\Rightarrow $ {\sc MinBas}$(G)=\{p_2,p_1,p_4\}$ is a minimal Janet basis for $\li F\ri$.
\end{exmp}

It is worth noting that in this example, we did not delete any polynomial by super top-reduction criterion (see Theorem \ref{g2vthm}).

\section{Experimental results}
\label{Exp}

We have implemented in {\tt Maple} 12\footnotemark\footnotetext{The {\tt Maple} codes of our programs and examples are available at {\tt http://invo.jinr.ru/.}} the algorithm {\sc InvComp} and the improved version~\cite{gerdtnew} of {\sc GBI} algorithm. For an efficient implementation of the last algorithm in {\tt Maple} we refer to \cite{blinkov}. It is worth noting that, in the given paper, we are willing to compare the structure and behavior of algorithms {\sc InvComp} and {\sc GBI} as they are implemented on the same platform. Therefore, we do not compare our implementations with \cite{blinkov}.  For experimental comparison of behavior of these algorithms, we used some well-known examples from the collection of benchmarks \cite{Bini} that has been already widely used for verification and comparison of different software packages created for construction of Gr\"obner bases.

The results are shown in the following tables. Table 1 compares the algorithms for Janet division, i.e. the $\succ_{\lex}$-division defined by (\ref{pair})-(\ref{inv_div}) with $\sqsupset$ being the pure lexicographical monomial ordering $\succ_{\lex}$ such that $x_i\succ_{\lex}x_j$ for $i>j$ and with $\sigma$ being the identical permutation. Table 2 shows the results of comparison for $\succ_{\alex}$-division under the same ordering on the variables as for the Janet division and also for the identical permutation $\sigma$. Here $\succ_{\alex}$ is the antigraded lexicographical monomial ordering \cite{CLO:2005:UAG} for which monomials $u$ and $v$ are compared in (\ref{inv_div}) as follows
\[
u\succ_{\alex} v \Longleftrightarrow \deg(u)<deg(v)\ \vee\ deg(u) = deg(v)\ \wedge\ u\succ_{\lex}v\,.
\]

The involutive bases computation was performed on a personal computer  with  $3.2$GHz,  2$\times$Intel(R)-Xeon(TM) Quad core, $24$ GB RAM and $64$ bits running  under the Linux operating system.  All computations were done over $\mathbb{Q}$, and for the input degree-reverse-lexicographical monomial ordering.

The {\it time} (resp. {\it memo.}, {\it reds.}, $C_1$ and $C_2$) column shows the consumed CPU time in seconds (resp. the amount of megabytes of memory used, the number of zero ${\L}$-normal forms computed,  the number of polynomials removed by $C_1$ and $C_2$ criteria) by the corresponding algorithm. In the seventh column the number of polynomials eliminated by the F$_5$ criterion is given. The eighth column represents the number of polynomials eliminated by {\em super top-reduction} criterion, denoted by $S$ which is applied as follows. Let $p$ be a quadruple. If $\LM(\poly(p))$ is divisible by the leading monomial of the polynomial part of some quadruple $q$, and $\LM(\poly(p))/\LM(\poly(q))\sig(q)= \sig(p)$, then we can discard $p$ by Theorem \ref{g2vthm}. The {\em polys.} column contains the number of polynomials in the involutive basis computed by the {\bf while}-loop in {\sc InvComp} (resp. outputted by {\sc GBI}). The last column {\it deg.} shows the largest degree of polynomials processed  during computation of involutive bases.

\vskip 0.5cm
\begin{center}
{\bf Table 1.} Benchmarking of {\sc InvComp} and {\sc GBI} for Janet division
\begin{table}[ht]
\centering {\tiny
\begin{tabular}{|c||c|c|c|c|c|c|c|c|c|}
\cline{1-10}
Cyclic$5$  & time  & memo. & reds.  & $C_1$ & $C_2$ &F$_5$ &S & polys. & deg.\\
\cline{1-10}
\noalign{\vskip 2pt}
\cline{1-10}
{\sc InvComp}  & 3.08 & 26.3 & 0  & 50 & 3 & 62  &44 & 52 & 9\\
\cline{1-10}
{\sc GBI}  &   22.60 & 164.60 & 83  & 40 & 5 &- & - & 23 &8\\
\cline{1-10}
\multicolumn{1}{c}{}  \\
\cline{1-10}
Weispfenning   & time  & memo. & reds.  & $C_1$ & $C_2$ &F$_5$ &S & polys. & deg.\\
\cline{1-10}
\noalign{\vskip 2pt}
\cline{1-10}
{\sc InvComp}  & 7.82 & 66.8 & 4  & 0 & 6 & 24  &68 & 67& 15\\
\cline{1-10}
{\sc GBI}  &   20.62 & 161.1 & 29  & 0 & 9 &- & - & 34 & 14\\
\cline{1-10}
\multicolumn{1}{c}{}  \\
\cline{1-10}
Haas$3$    & time  & memo. & reds.  & $C_1$ & $C_2$ &F$_5$ &S & polys. & deg.\\
\cline{1-10}
\noalign{\vskip 2pt}
\cline{1-10}
{\sc InvComp}  & 18.56 & 161.9 & 0  & 0 & 25 & 98  &154 & 150 & 13\\
\cline{1-10}
{\sc GBI}  &   61.85 & 475.1 & 121  & 0 & 11 &- & - & 73& 12\\
\cline{1-10}
\multicolumn{1}{c}{}  \\
\cline{1-10}
Katsura$5$   & time  & memo. & reds.  & $C_1$ & $C_2$ &F$_5$ &S & polys. & deg.\\
\cline{1-10}
\noalign{\vskip 2pt}
\cline{1-10}
{\sc InvComp}  & 56.00 & 495.5 & 0  & 98 & 22 & 138  & 147 & 113& 8\\
\cline{1-10}
{\sc GBI}  & 25.52 & 207.1 & 47  & 22 & 1 &- & - & 23 & 6\\
\cline{1-10}
\multicolumn{1}{c}{}  \\
\cline{1-10}
Lichtblau    & time  & memo. & reds.  & $C_1$ & $C_2$ &F$_5$ &S & polys. & deg.\\
\cline{1-10}
\noalign{\vskip 2pt}
\cline{1-10}
{\sc InvComp}  & 229.87 & 1892.7 & 0  & 0 & 109 & 43  & 296 & 271 & 19\\
\cline{1-10}
{\sc GBI}  & $>8$ hours & ? & ?  & ? & ? &- & - &? & ?\\
\cline{1-10}
\multicolumn{1}{c}{}  \\
\cline{1-10}
Cyclic$6$   & time  & memo. & reds.  & $C_1$ & $C_2$ &F$_5$ &S & polys. &deg.\\
\cline{1-10}
\noalign{\vskip 2pt}
\cline{1-10}
{\sc InvComp}  & 405.20 & 4122.3 & 13  & 246 & 111 & 361  & 607 & 297 & 11\\
\cline{1-10}
{\sc GBI}  & 5208.32 & 89559.9 & 476  & 152 & 18 &- & - & 46 & 10\\
\cline{1-10}
\multicolumn{1}{c}{}  \\
\cline{1-10}
Katsura$6$   & time  & memo. & reds.  & $C_1$ & $C_2$ &F$_5$ &S & polys. & deg.\\
\cline{1-10}
\noalign{\vskip 2pt}
\cline{1-10}
{\sc InvComp}  & 739.80 & 5471.1 & 0  & 165 & 104 & 222  & 274 & 205 & 11\\
\cline{1-10}
{\sc GBI}  & 5168.90 & 205361.4 & 128  & 44 & 3 &- & - & 43 & 7\\
\cline{1-10}
\multicolumn{1}{c}{}  \\
\cline{1-10}
Eco$7$   & time  & memo. & reds.  & $C_1$ & $C_2$ &F$_5$ &S & polys. & deg.\\
\cline{1-10}
\noalign{\vskip 2pt}
\cline{1-10}
{\sc InvComp}  & 2492.10 & 24639.4 & 0  & 199 & 936 & 557  & 460 & 459 & 12\\
\cline{1-10}
{\sc GBI}  & 102.14 & 947.0 &124  & 46 & 40 &- & - & 45 &4\\
\cline{1-10}
\multicolumn{1}{c}{}  \\
\end{tabular}}
\label{table1}
\end{table}
\end{center}

As one can see from the column {\it reds.}, Buchberger's criteria $C_1$, $C_2$ together with the F$_5$ criterion and the super top-reduction criterion $S$ do detect the vast majority of useless zero reductions whereas the criteria $C_1$, $C_2$ do not. In so doing, for the Janet division (Table 1) only in the two examples of eight there are some undetected zero reductions whereas in the case of $\succ_{\alex}$-division (Table 2) there is one half of such examples. The price one has to pay for this extra detection in {\sc InvComp} in comparison with {\sc GBI} is a more lengthy intermediate basis $T$ (cf. the numbers in column {\it polys.}). For the Janet bases in Table 1, except two examples, this enlargement in combination with the extra detection of useless reductions leads to faster computation (column time) correlated with less memory consumed (column {\it memo.}). For the $\succ_{\alex}$-division the enlargement of the intermediate set $T$ in {\sc InvComp} and  respectively the maximal total degree of its polynomials (see Table 2, columns {\it polys.} and {\it deg.}) are substantially larger and, except one example, is not compensated (in comparison with {\sc GBI}) by the additional detection of zero reductions.

\vskip 0.5cm
\begin{center}
{\bf Table 2.} Benchmarking of {\sc InvComp} and {\sc GBI} for $\succ_{\alex}$-division
\begin{table}[ht]
\centering {\tiny
\begin{tabular}{|c||c|c|c|c|c|c|c|c|c|}
\cline{1-10}
Wang$89$   & time  & memo. & reds.  & $C_1$ & $C_2$ &F$_5$ &S & polys. & deg.\\
\cline{1-10}
\noalign{\vskip 2pt}
\cline{1-10}
{\sc InvComp}  & 11.19 & 90.6 & 0  & 0 & 0 & 66  & 67 & 67& 14\\
\cline{1-10}
{\sc GBI}  & 0.71 & 6.9 & 12  & 0 & 0 &- & - & 10 & 7\\
\cline{1-10}
\multicolumn{1}{c}{}  \\
\cline{1-10}
Cyclic$5$  & time  & memo. & reds.  & $C_1$ & $C_2$ &F$_5$ &S & polys. & deg.\\
\cline{1-10}
\noalign{\vskip 2pt}
\cline{1-10}
{\sc InvComp}  & 15.14 & 73.9 & 0  & 44 & 40 & 57  &49 & 56 & 16\\
\cline{1-10}
{\sc GBI}  &   26.25 & 177.9 & 83  & 40 & 3 &- & - & 23 & 8\\
\cline{1-10}
\multicolumn{1}{c}{}  \\
\cline{1-10}
Gerdt$2$   & time  & memo. & reds.  & $C_1$ & $C_2$ &F$_5$ &S & polys. & deg.\\
\cline{1-10}
\noalign{\vskip 2pt}
\cline{1-10}
{\sc InvComp}  & 20.48 & 145.8 & 66  & 0 & 2 & 3  & 114 & 55 & 14\\
\cline{1-10}
{\sc GBI}  & 0.56 & 4.4 & 4  & 0 & 0 &- & - & 8 & 6\\
\cline{1-10}
\multicolumn{1}{c}{}  \\
\cline{1-10}
Pavelle    & time  & memo. & reds.  & $C_1$ & $C_2$ &F$_5$ &S & polys. & deg.\\
\cline{1-10}
\noalign{\vskip 2pt}
\cline{1-10}
{\sc InvComp}  & 112.22 & 804.6 & 0  & 0 & 37 & 156  & 252 & 139 & 11\\
\cline{1-10}
{\sc GBI}  & 1.88 & 15.5 & 18  & 0 & 0 &- & - & 12 & 4\\
\cline{1-10}
\multicolumn{1}{c}{}  \\
\cline{1-10}
Trinks    & time  & memo. & reds.  & $C_1$ & $C_2$ &F$_5$ &S & polys. & deg.\\
\cline{1-10}
\noalign{\vskip 2pt}
\cline{1-10}
{\sc InvComp}  & 372.06 & 2908.5 & 3  & 94 & 377 & 139  & 536 & 224 & 13\\
\cline{1-10}
{\sc GBI}  &   28.78 & 178.5 & 16  & 26 & 112 &- & - & 38& 8\\
\cline{1-10}
\multicolumn{1}{c}{}  \\
\cline{1-10}
Weispfenning   & time  & memo. & reds.  & $C_1$ & $C_2$ &F$_5$ &S & polys. & deg.\\
\cline{1-10}
\noalign{\vskip 2pt}
\cline{1-10}
{\sc InvComp}  & 800.16 & 3898.5 & 4  & 0 & 27 & 37  &900 & 204& 22\\
\cline{1-10}
{\sc GBI}  &   432.76 & 3112.6 & 29  & 0 & 116 &- & - & 92 & 21\\
\cline{1-10}
\multicolumn{1}{c}{}  \\
\cline{1-10}
Liu   & time  & memo. & reds.  & $C_1$ & $C_2$ &F$_5$ &S & polys. & deg.\\
\cline{1-10}
\noalign{\vskip 2pt}
\cline{1-10}
{\sc InvComp}  & 4568.69 & 22815.1 & 0  & 6 & 64 & 489  & 1048 & 451 & 15\\
\cline{1-10}
{\sc GBI}  & 2.43 & 17.7 & 18  & 0 & 0 &- & - & 12 &5\\
\cline{1-10}
\multicolumn{1}{c}{}  \\
\cline{1-10}
Cyclic$6$   & time  & memo. & reds.  & $C_1$ & $C_2$ &F$_5$ &S & polys. &deg.\\
\cline{1-10}
\noalign{\vskip 2pt}
\cline{1-10}
{\sc InvComp}  & 49301.16 & 229416.4 & 76  & 430 & 1895 & 1419  & 5890 & 959 & 22\\
\cline{1-10}
{\sc GBI}  & 5184.06 & 100458.6 & 458  & 147 & 3 &- & - & 46 & 10\\
\cline{1-10}
\multicolumn{1}{c}{}  \\
\end{tabular}}
\label{table2}
\end{table}
\end{center}

Heuristically, as it was shown in~\cite{alex}, the $\succ_{\alex}$-division is not worse than the Janet division w.r.t. the number of nonmultiplicative prolongations to be processed in the course of completion to involution. Therefore, the main reason of slowdown of the algorithm {\sc InvComp} for $\succ_{\alex}$-division, as compared with the Janet division, is to be the presence in intermediate basis $T$ of the multiplicative prolongations of its elements caused by the enlargement of $Q$ done in line 26 of {\sc InvComp}. By Corollary \ref{cor_mult}, the last enlargement is not done for Janet bases.
The data in Tables 1 and 2 for examples Cyclic5 and Cyclic6 nicely illustrate this distinction in behavior of the two divisions. Both minimal Janet and $\succ_{\alex}$-bases for every of these examples have the same number of elements. At the same time the {\bf while}-loop of {\sc Invcomp} outputs a much larger $\succ_{\alex}$-basis than the corresponding Janet basis. In doing so, the cardinality of the $\succ_{\alex}$-basis for Cyclic6 is more than three times higher than the cardinality of the Janet basis and the maximal degree of the former is twice higher than that of the latter.

The next two figures illustrate an experimental comparison of the memory used and the time taken by the algorithms {\sc InvComp} and {\sc InvComp} for the Janet division (Fig.~1) and the $\succ_{\alex}$-division (Fig.~2).

\begin{figure}[h]
\begin{center}
{\includegraphics[scale=0.22]{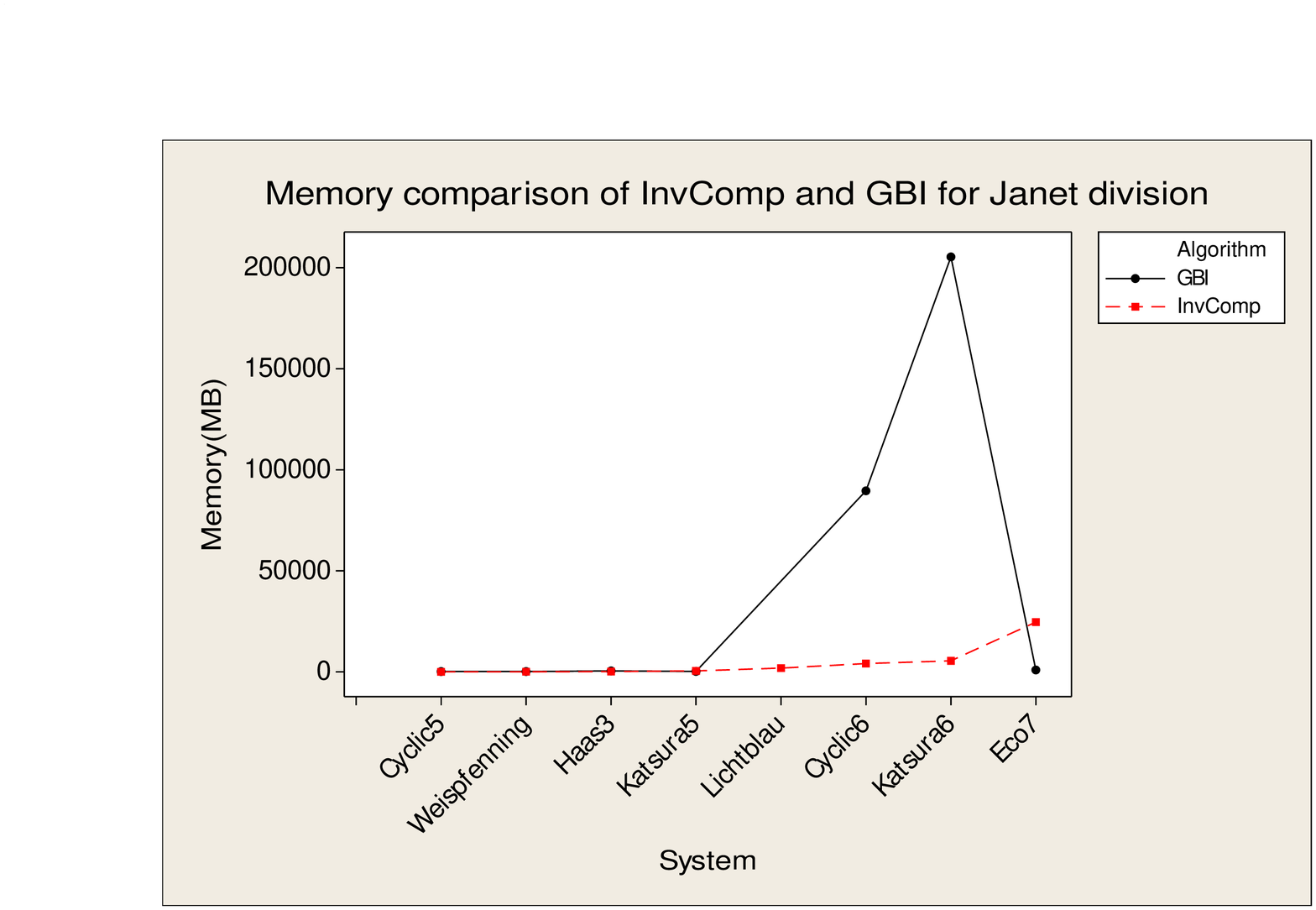}}
{\includegraphics[scale=0.22]{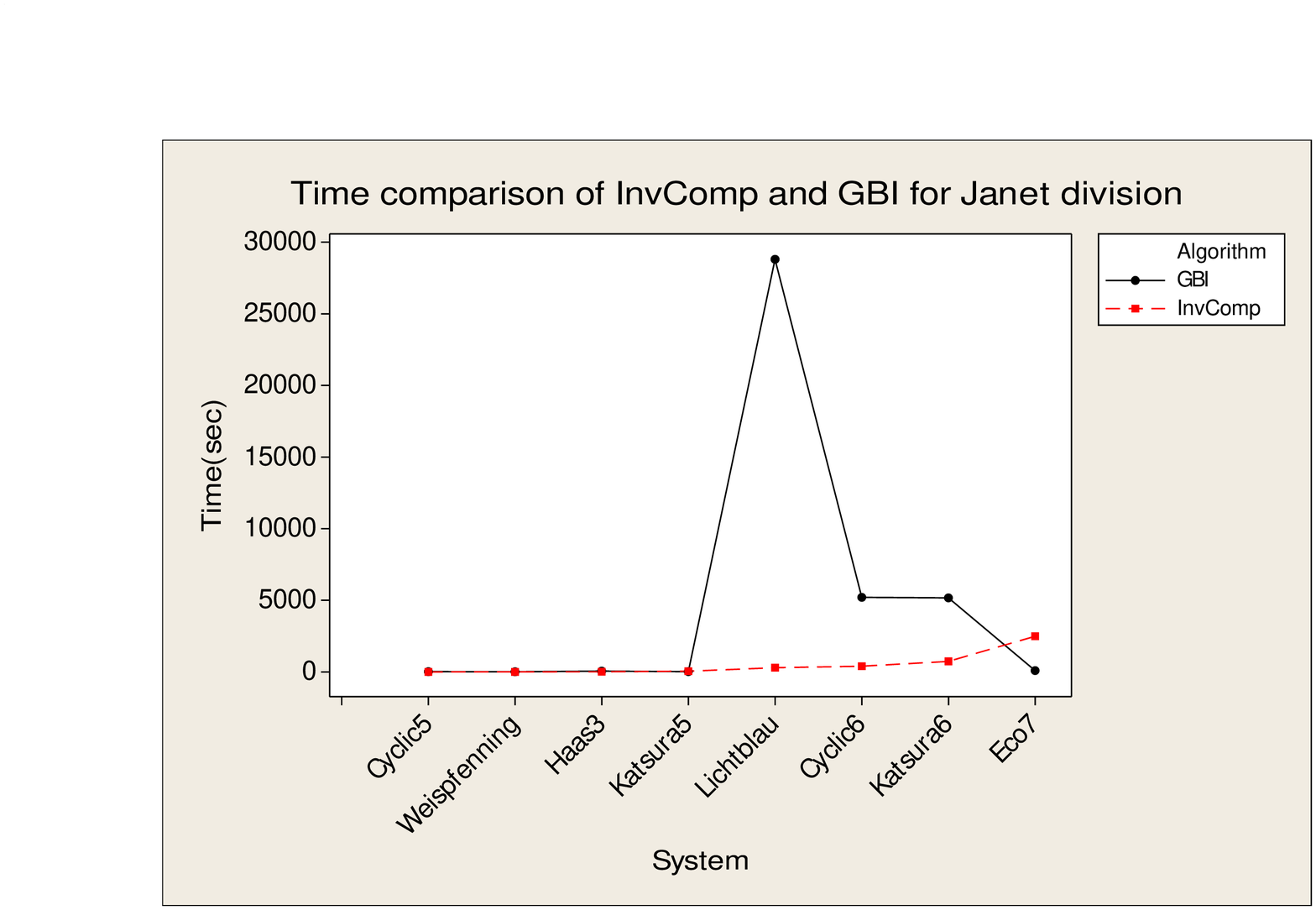}}
\caption{Comparison of {\sc InvComp} and {\sc GBI} for Janet division}
\end{center}
\end{figure}

\begin{figure}[ht]
\begin{center}
{\includegraphics[scale=0.22]{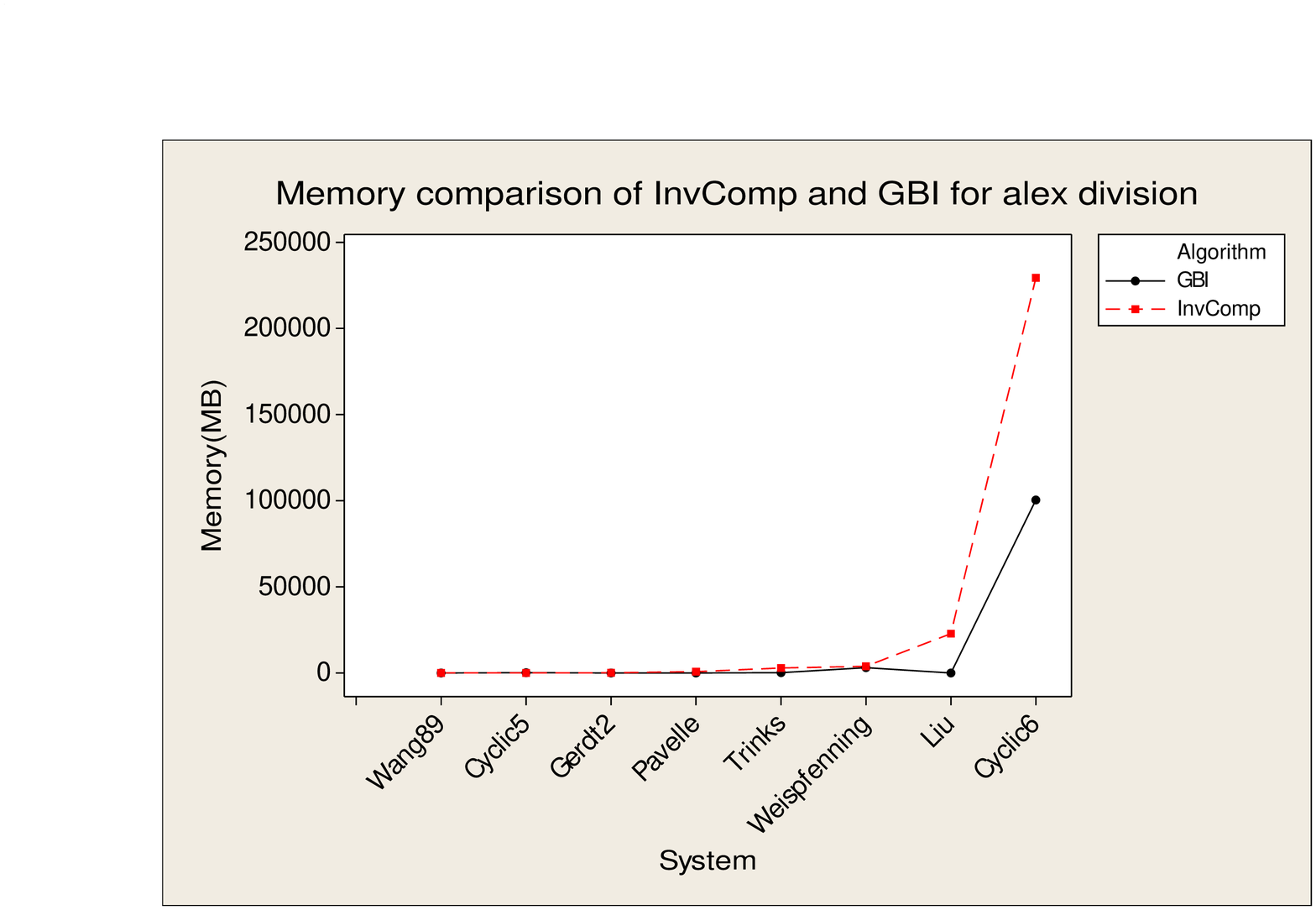}}
{\includegraphics[scale=0.22]{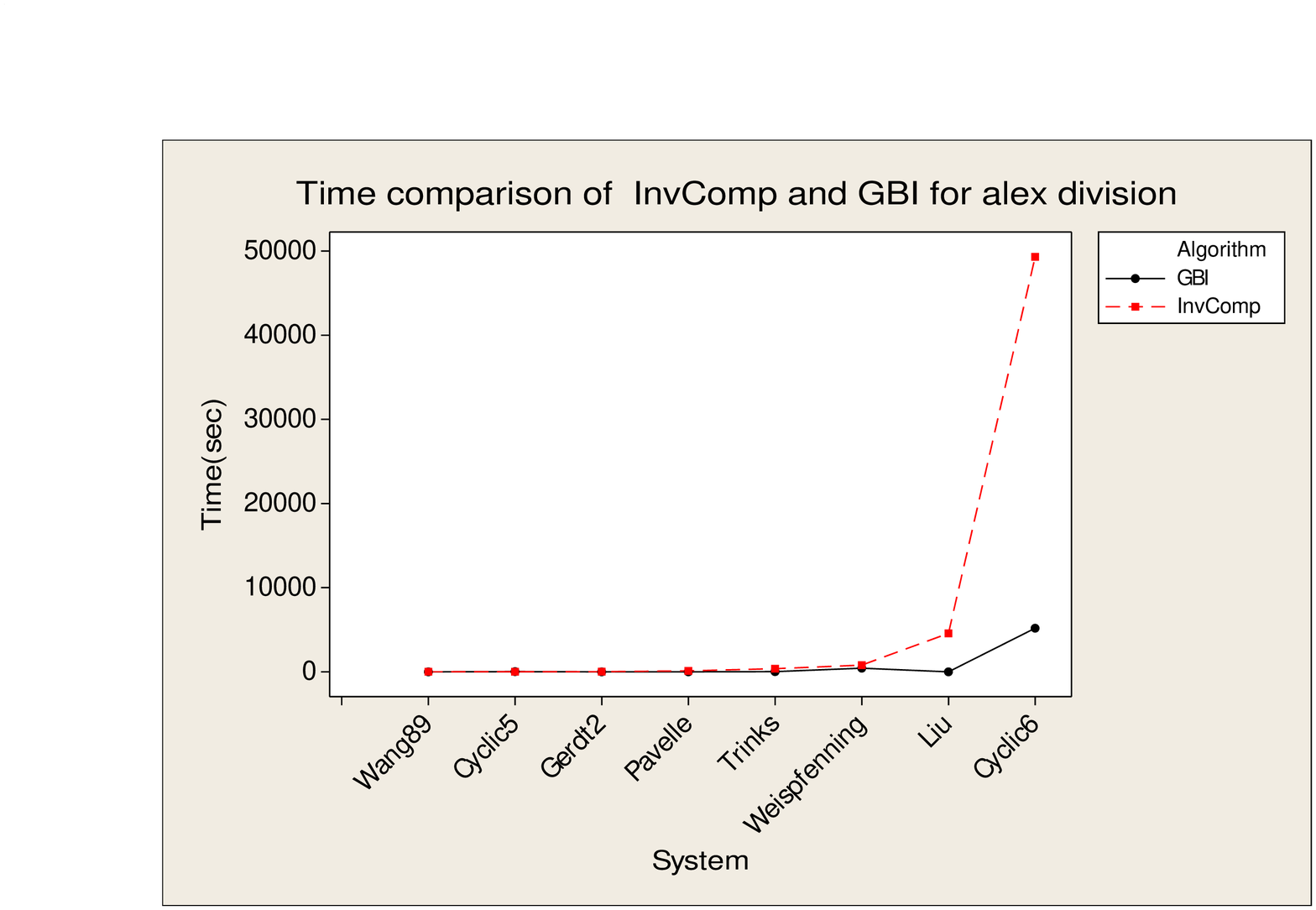}}
\caption{Comparison of {\sc InvComp} and {\sc GBI} for $\succ_{\alex}$-division}
\label{graph}
\end{center}
\end{figure}

It should be noted that the above presented experimental analysis of our new involutive completion algorithm {\sc InvComp} is underdrawn. One needs to implement it efficiently either in {\tt Maple} and compare with implementation of the GBI algorithm done in {\cite{blinkov} or in C/C++ and compare with the {\tt GINV} software~\cite{ginv}. In the last case the choice of heuristically good selection strategy~\cite{GB07} for a nonmultiplicative polongation to be processed (cf. line 5 in algorithm {\sc InvBase}) and the use of proper data structures for the fast search of an involutive divisor \cite{gerdtnew} play a key role for the efficiency of involutive bases computation. As it was demonstrated by Faug{\`e}re \cite{F4,F5}, another very important source for computational efficiency of \Gr bases algorithms is a clever use of linear algebra for performing reductions. We believe that this is applies equally to the involutive algorithms. Experimental comparison of {\tt GINV} and another GBI implementation ({\tt JB}) with  {\tt Magma}, {\tt Singular} and accessible implementations of signature-based algorithms which do not exploit linear algebra is given on the Web page~{\tt http://cag.jinr.ru/wiki/Benchmarking\_for\_polynomial\_ideals}.

\section*{Acknowledgements.}  The authors thank the anonymous referees for constructive comments and recommendations which helped to improve the readability and quality of the paper. The authors also thank Daniel Robertz for helpful remarks. The first author (V.P.G.) was initially motivated in incorporation of the F$_5$ criterion into involutive algorithms during his stay at the Laboratory of Computer Science of University Pierre and Marie Curie in May 2011. He is grateful to Jean-Charles Faug{\`e}re for supporting that visit and for stimulating discussions. The research presented in the given paper was performed during the stay of the second author (A.H.) at Joint Institute for Nuclear Research in Dubna, Russia. The contribution of the first author (V.P.G.) was partially supported by the grants 12-07-00294 and 13-01-00668 from the Russian Foundation for Basic Research and by the grant 3802.2012.2 from the Ministry of Education and Science of the Russian Federation.

\bibliographystyle{alpha}

\end{document}